\documentclass{amsart}
\usepackage{mathtools} % \coloneqq for the definition symbol ':='.
\usepackage{amssymb}
\usepackage[mathcal]{euscript}
\usepackage[cmtip,all]{xy}
\usepackage{amsmath}\usepackage{amsfonts}%% letters and caracters.
\usepackage{graphicx} % scalebox,...
\usepackage{verbatim} % \begin{comment} ...
\usepackage{stackrel}

\DeclareMathOperator{\End}{\mathrm{End}}

\DeclareMathOperator{\Alg}{\mathrm{Alg}}

\DeclareMathOperator{\Lie}{\mathrm{Lie}}
\DeclareMathOperator{\Der}{\mathrm{Der}}

\DeclareMathOperator{\Aut}{\mathrm{Aut}}
\DeclareMathOperator{\AAut}{\mathbf{Aut}}

\DeclareMathOperator{\Hom}{\mathrm{Hom}}

\DeclareMathOperator{\im}{\mathrm{im}}
\DeclareMathOperator{\Supp}{\mathrm{Supp}}
\DeclareMathOperator{\ch}{\mathrm{char}}
\DeclareMathOperator{\dg}{\mathrm{deg}}
\DeclareMathOperator{\dm}{\mathrm{dim}}

\newcommand{\R}{\mathbb{R}}  % Reals
\newcommand{\Z}{\mathbb{Z}}  % Integers
\newcommand{\F}{\mathbb{F}}  % Field

\newcommand{\cA}{\mathcal{A}} % Algebras
\newcommand{\cB}{\mathcal{B}}
\newcommand{\cH}{\mathcal{H}}

\newcommand{\cK}{\mathcal{K}}

\newcommand{\cV}{\mathcal{V}}

\newcommand{\cC}{\mathcal{C}}
\newcommand{\cS}{\mathcal{S}}

 \newcommand{\FF}{\mathbb{F}} \newcommand{\ZZ}{\mathbb{Z}}  \newcommand{\OO}{\mathbb{O}}
\newcommand{\med}{\;|\;}

\DeclareMathOperator{\chr}{\mathrm{char}\,}  \DeclareMathOperator{\lspan}{\mathrm{span}} 
\DeclareMathOperator{\kan}{\mathfrak{K}}
\newcommand{\Univ}{\mathcal{U}} 

\newcommand{\smir}{{\textmd{\rm T}(\cC)}}
\newcommand{\smirCxC}{{\textmd{\rm T}(\cC\otimes\cC)}}
\newcommand{\TT}{\textmd{\rm T}}
\DeclareMathOperator{\tr}{\mathrm{t}} %trace
 \newcommand{\CxC}{{\cC\otimes\cC}}

\newtheorem{theorem}{Theorem}
\newtheorem{proposition}[theorem]{Proposition}
\newtheorem{lemma}[theorem]{Lemma}
\newtheorem{corollary}[theorem]{Corollary}

\theoremstyle{definition}
\newtheorem{df}[theorem]{Definition}
\newtheorem{example}[theorem]{Example}

\theoremstyle{remark}
\newtheorem{remark}[theorem]{Remark}

\numberwithin{equation}{section} %sets equation numbers <chapter>.<section>.<index>
\numberwithin{theorem}{section} %sets equation numbers <chapter>.<section>.<index>
\begin{document}
%%%%%%%%%%%%%%%%%%%%%%%%%%%%%%%%%%%%%%%%%%%%%%%%%%%%%%%%%%%%%%%%%%%%%%%%%%%%%%%%%%%%%%%%%%%%%%%%%%%%

\title[Grads. on tensor prod. of comp. algs. and on the Smirnov algebra]
{Gradings on tensor products of composition algebras and on the Smirnov algebra}

\author[D. Aranda-Orna]{Diego Aranda-Orna${}^\star$}
\address{Departamento de Matem\'{a}ticas
 e Instituto Universitario de Matem\'aticas y Aplicaciones,
 Universidad de Zaragoza, 50009 Zaragoza, Spain}
\email{diego.aranda.orna@gmail.com}

\author[A.S. C\'ordova-Mart\'inez]{Alejandra S. C\'ordova-Mart\'inez${}^\star$}
\address{Departamento de Matem\'{a}ticas
 e Instituto Universitario de Matem\'aticas y Aplicaciones,
 Universidad de Zaragoza, 50009 Zaragoza, Spain}
\email{sarina.cordova@gmail.com}

\thanks{${}^\star$Supported by the Spanish Ministerio de Econom\'{\i}a y 
Competitividad---Fondo Europeo de Desarrollo Regional (FEDER) MTM2017--83506-C2-1-P.\quad  
A.S.~C\'ordova-Mart{\'\i}nez also acknowledges support from the Consejo Nacional de Ciencia y Tecnolog{\'\i}a (CONACyT, M\'exico) through grant 420842/262964.}

\date{}

\begin{abstract}
We give classifications of group gradings, up to equivalence and up to isomorphism, on the tensor product of a Cayley algebra $\cC$ and a Hurwitz algebra over a field of characteristic different from 2. We also prove that the automorphism group schemes of $\cC^{\otimes n}$ and $\cC^n$ are isomorphic.

On the other hand, we prove that the automorphism group schemes of a Smirnov algebra $\smir$ (a $35$-dimensional simple exceptional structurable algebra constructed from a Cayley algebra $\cC$) and $\cC$ are isomorphic. This is used to obtain classifications, up to equivalence and up to isomorphism, of the group gradings on Smirnov algebras.
\end{abstract}

\maketitle
%%%%%%%%%%%%%%%%%%%%%%%%%%%%%%%%%%%%%%%%%%%%%%%%%%%%%%%%%%%%%%%%%%%%%%%%%%%%%%%%%%%%%%%%%%%%

\section{Introduction}

A classification of finite-dimensional central simple structurable algebras over a field of
characteristic zero was given in 1978 in \cite[Theorem 25]{Al78}, with a missing case. Such classification was completed in 1990 (see \cite{Smi90a} and \cite{Smi92}) for a base field of characteristic different from 2, 3 and 5. Allison and Faulkner extended the definition of Structurable algebras to arbitrary rings of scalars of any characteristic \cite[\S 5]{AF93a}.
The importance of studying structurable algebras is their use in the construction of Lie algebras using, for example, a modified TKK-construction as in \cite{All79} where all the isotropic simple Lie algebras were obtained over an arbitrary field of characteristic zero.
From a $G$-grading on a central simple structurable algebra, where $G$ is a group, we can get a $G \times \Z$-grading on its corresponding central simple Lie algebra. We are interested in two cases of the classification: the tensor product of a Cayley algebra $\cC$ and a Hurwitz algebra, and the Smirnov algebra $\smir$. Note that in the case of $\cC \otimes \FF \cong \cC$, the classification of group gradings is well-known (\cite{Eld98}).

We know, by \cite{All79}, that we can obtain the central simple Lie algebras of type $F_4$, $E_6$, $E_7$ and $E_8$ through a construction related with the mentioned one from the algebras $(\cC\otimes\cB,-)$ for a Cayley algebra $\cC$ and a Hurwitz algebra $\cB$, where $-$ is the tensor product of their involutions.

\bigskip

By grading we mean group grading. We will always assume that the characteristic of the base field is different from $2$. This paper is structured as follows.
 
In Section~\ref{section.preliminaries} we recall the basic definitions and well-known results used in the rest of the paper.

In Section~\ref{section.Tensor} we first prove that the automorphism group schemes $\AAut(\cC^n)$, $\AAut (\cC^{\otimes n})$ and $\AAut (\cC^{\otimes n}, \bigotimes_{i=1}^n -)$
are isomorphic, where $\cC$ is the Cayley algebra and $-$ the standard involution. Then we give a classification of (involution preserving) gradings on the tensor product of a Cayley algebra and a Hurwitz algebra.

In Section~\ref{section.smirnov} we prove that the automorphism group schemes of a Smirnov algebra $\smir$ and its associated Cayley algebra $\cC$ are isomorphic. It is used for classifying the gradings, up to equivalence and up to isomorphism, on Smirnov algebras. 

Finally, in Section~\ref{section.induced} we show how the gradings on the structurable algebras considered in this paper can be used to induce gradings on Lie algebras via several constructions.

\section{Preliminaries} \label{section.preliminaries}
%%%%%%%%%%%%%%%%%%%%%%%%%%%%%%%%%%%%%%%%%%%%%%%%%%%%%%%%%%%%%%%%%%%%%%%%%%%%%%%%%%%%%%
%%%%%%%%%%%%%%%%%     *SUBSECTION*   GRADINGS             %%%%%%%%%%%%%%%%%%%%%%%%%%%
%%%%%%%%%%%%%%%%%%%%%%%%%%%%%%%%%%%%%%%%%%%%%%%%%%%%%%%%%%%%%%%%%%%%%%%%%%%%%%%%%%%%%%
\subsection{Gradings}\label{subsection.gradings}

\begin{df}\label{df:0.1}

A \textit{grading} by a group $G$ on an algebra $\mathcal{A}$ (not necessarily associative) over a field $\F$, or a \textit{$G$-grading} on $\mathcal{A}$, is a vector space decomposition
\[
\Gamma: \mathcal{A} = \bigoplus\limits_{g \in G} \mathcal{A}_g
\]
satisfying $\mathcal{A}_{g} \mathcal{A}_{h} \subset \mathcal{A}_{g h}$ for all $g, h \in G$. If such a decomposition is fixed we will refer to $\mathcal{A}$ as a $G$-graded algebra. The set
\[
\Supp \Gamma:= \lbrace g\in G : \cA_g \neq 0 \rbrace
\]
is called the \textit{support} of $\Gamma$. A grading is \textit{nontrivial} if the support  consists of more than one element. If $0\neq a \in \cA_g$, then we say that $a$ is \textit{homogeneous} of \textit{degree} $g$ and we write $\dg_\Gamma a=g$, or just $\dg a=g$ when the associated grading is clear. The subspace $\cA_g$ is called the \textit{homogeneous component} of degree $g$. A \textit{(vector space) grading} on a vector space $V$ is a grading on the algebra given by $V$ with the trivial product.

A subspace (resp. subalgebra) $V \subset \cA$ is said to be a \textit{graded subspace} (resp. \textit{graded subalgebra}) if  
\[
V = \bigoplus\limits_{g \in G} (\cA_g \cap V).
\]
Taking $V_g = \cA_g \cap V$, we turn $V$ into a $G$-graded vector space (resp. algebra). A \textit{graded ideal} is an ideal which is a graded subspace.
\end{df}

\begin{df}\label{df:graded.simple}
Let $\mathcal{A}$ be an algebra. If $\mathcal{A}$ is a $G$-graded algebra we say that 
$\mathcal{A}$ is $G$-\textit{graded-simple} if $\mathcal{A} \mathcal{A} \neq 0$ and the only graded ideals of $\cA$ are $\lbrace 0 \rbrace$ and $\mathcal{A}$. When it is clear which the grading group is we simply write ``graded-simple''.
\end{df}

For a grading $\Gamma$ we can consider many grading groups but, there is one distinguished grading group called \textit{universal group}, denoted by $U(\Gamma)$ (\cite[Chapter 1.2]{EKmon}).

We will now recall two natural ways to define an equivalence relation on group gradings, depending on whether the grading group plays a secondary role or not.

\begin{df} \label{df:1.14}
Let $\Gamma$ be a $G$-grading on an algebra $\cA$ and let $\Gamma'$ be an $H$-grading on an algebra $\cB$. We say that $\Gamma$ and $\Gamma'$ are \textit{equivalent} if there exist an isomorphism of algebras $\varphi: \mathcal{A} \rightarrow \mathcal{B}$ and a bijection $\alpha: \Supp \Gamma \rightarrow \Supp \Gamma'$ such that $\varphi (\mathcal{A}_g)= \mathcal{B}_{\alpha(g)}$ for all $g \in \Supp \Gamma$.
\end{df}

\begin{df}\label{df:1.15}
Let $\Gamma$ and $\Gamma'$ be two $G$-gradings on the algebras $\mathcal{A}$ and $\mathcal{B}$, respectively. We say that $\Gamma$ and $\Gamma'$ are \textit{isomorphic} if there exists an isomorphism of algebras $\varphi: \mathcal{A} \rightarrow \mathcal{B}$ such that $\varphi (\mathcal{A}_g)= \mathcal{B}_g$ for all $g \in G$. 
\end{df}

\begin{df}\label{df:1.24}
Let $\Gamma: \mathcal{A} = \bigoplus_{g\in G} \mathcal{A}_g$ and $\Gamma':\mathcal{A} = \bigoplus_{h\in H} \mathcal{A}'_h$ be two gradings. We say that $\Gamma$ is a \textit{refinement} of $\Gamma'$, or that $\Gamma'$ is a \textit{coarsening} of $\Gamma$, if for any $g \in G$ there exists $h\in H$ such that $\mathcal{A}_g \subseteq \mathcal{A}'_h$. If, for some $g \in G$, the inclusion is strict, then we say that we have a \textit{proper} refinement or coarsening. We say $\Gamma$ is \textit{fine} if it does not admit proper refinements.
\end{df}

The study of group gradings on finite-dimensional algebras is reduced to the study of fine gradings by their universal groups on such algebras (\cite[Proposition 1.25, Corollaries 1.26 and 1.27]{EKmon})

The next definition will be used in the process of obtaining gradings on the direct product of two Cayley algebras.

\begin{df}\label{df:loop} \cite[Definition 3.1.1]{ABFP08}

Let $\pi: G \rightarrow \overline{G}$ be a group epimorphism of abelian groups and let $\cA$ be an algebra. Denote $\pi (g) = \overline{g}$ for $g \in G$.
Suppose that there is a $\overline{G}$-grading $\overline{\Gamma} : \mathcal{A} = \bigoplus_{\overline{g}
\in \overline{G}} \mathcal{A}_{\overline{g}}$. 
Then
$$L_{\pi}(\mathcal{A})= \sum\limits_{g \in G} \mathcal{A}_{\overline{g}} \otimes g \hspace*{0.2cm} (\leq\mathcal{A} \otimes_{\F} \F G)$$
is a $G$-graded algebra with $L_{\pi}(\mathcal{A})_{g} = \mathcal{A}_{\overline{g}} \otimes g$ for $g \in G$. This algebra is called the \textit{loop algebra} of $\mathcal{A}$ relative to $\pi$. 

\end{df}

%%%%%%%%%%%%%%%%%%%%%%%%%%%%%%%%%%%%%%%%%%%%%%%%%%%%%%%%%%%%%%%%%%%%%%%%%
%%%%%%%%%%%%%%%%%%%      SCHEMES      %%%%%%%%%%%%%%%%%%%%%%%%%%%%%%%%%%%
%%%%%%%%%%%%%%%%%%%%%%%%%%%%%%%%%%%%%%%%%%%%%%%%%%%%%%%%%%%%%%%%%%%%%%%%%

\subsection{Schemes}\label{subsection.schemes}

The following can be found in \cite[Appendix A]{EKmon}. 

\begin{df}\label{df:scheme}
An \textit{affine group scheme} over $\F$ is a representable functor from the category $\Alg_{\F}$ of commutative associative unital algebras over a field $\F$ to the category of groups. 

Let $\textbf{G}$ and $\textbf{H}$ be affine group schemes. We say that $\textbf{H}$ is a \textit{subgroupscheme} of $\textbf{G}$ if, for any object $R$ in $\Alg_{\F}$, the group $\textbf{H}(R)$ is a subgroup of $\textbf{G}(R)$ and the injections $\textbf{H}(R) \hookrightarrow \textbf{G}(R)$ form a natural map $\textbf{H}\rightarrow \textbf{G}$.
\end{df}

Let $\mathcal{A}$ be a finite-dimensional nonassociative algebra over $\F$. The \textit{automorphism group scheme of $\cA$}, $\AAut(\mathcal{A})$, is defined by 
\[
\AAut(\mathcal{A})(R):= \Aut_R (\mathcal{A} \otimes R)
\] 
for any object $R$ in $\Alg_{\F}$.
If $G$ is an abelian group, a $G$-grading on an algebra $\mathcal{A}$ corresponds to a homomorphism of affine group schemes $G^D \longrightarrow \textbf{Aut}(\mathcal{A})$ (\cite[Proposition 1.36]{EKmon}), where $G^D$ is the Cartier dual of $G$. Therefore if $\cB$ is an algebra such that $\textbf{Aut}(\mathcal{A}) \simeq \textbf{Aut}(\mathcal{B})$, then there is a natural correspondence between $G$-gradings on $\cA$ and $G$-gradings on $\cB$.

A result we will use more than once is the following.

\begin{theorem}\label{th:A.50}\cite[Theorem A.50]{EKmon}
Let $\theta: \mathbf{G} \rightarrow \mathbf{H}$ be a morphism of affine algebraic group
schemes. Assume that $\mathbf{G}$ or $\mathbf{H}$ is smooth. Then $\theta$ is an isomorphism if and only if

1) $\theta_{\overline{\F}}: \mathbf{G}(\overline{\F}) \rightarrow \mathbf{H} (\overline{\F})$ is bijective and 

2) $d \theta: \Lie (\mathbf{G}) \rightarrow \Lie(\mathbf{H})$ is bijective.

\end{theorem}

%%%%%%%%%%%%%%%%%%%%%%%%%%%%%%%%%%%%%%%%%%%%%%%%%%%%%%%%%%%%%%%%%%%%%%%%%%%%%%%%%%%%%%
%%%%%%%%%%%%%%%%%     *SUBSECTION*   STRUCTURABLE ALGEBRAS             %%%%%%%%%%%%%%%
%%%%%%%%%%%%%%%%%%%%%%%%%%%%%%%%%%%%%%%%%%%%%%%%%%%%%%%%%%%%%%%%%%%%%%%%%%%%%%%%%%%%%%

\subsection{Structurable algebras}\label{subsection.structurable}

First we recall some definitions. Let $(\cA,-)$ be a unital algebra with involution. Define
$V_{x,y} \in \End_{\FF} (\cA)$ by
\[
V_{x,y}(z)=(x \bar{y})z+(z\bar{y})x-(z\bar{x})y
\]
for any $x,y,z$ in an algebra $\cA$. Put $T_x = V_{x,1}$, for any $x \in \cA$, that is,
\[ 
T_x (z)= xz + zx - z \overline{x}
\]
for any  $x,z \in \cA$.

\begin{df}\label{df:structurable.algebra} \cite{Al78}
Let $\FF$ be a field of characteristic different from 2 and 3. Let $(\cA,-)$ be a finite-dimensional nonassociative unital algebra with involution over $\FF$ (i.e., an antiautomorphism ``--'' of period 2). We say that $(\cA,-)$ is \textit{structurable} if
\[
[T_z, V_{x,y}] = V_{T_z x,y} - V_{x, T_{\overline{z} y}}
\]
for any $x,y,z \in \cA$. We denote by
$$\cH(\cA,-) = \{x\in\cA \med \bar x = x\} \quad \text{and} \quad \cS(\cA,-) = \{x\in\cA \med \bar x = -x\}$$
the subspaces of \textit{symmetric} (or \textit{hermitian}) and \textit{skew-symmetric} (or \textit{skew}) elements of $\cA$, respectively.
It is straightforward to prove that $\cS(\cA,-)$ is a non-unital subalgebra of $(\cA,-)$ with the multiplication given by the commutator $[\cdot,\cdot]$.

\end{df}

\begin{df}\label{df:structurable_grading}
Let $G$ be a group and $(\cA,-)$ an algebra with involution. We say that $\Gamma$ is an \textit{involution preserving grading} on $(\cA,-)$ if $\Gamma$ is a $G$-grading on the algebra $\cA$ and it is closed under the involution, i.e., $\overline{\cA}_g \subseteq \cA_g$ for all $g\in G$.

Let $(\cA,-_{\cA})$ and $(\cB,-_{\cB})$ be algebras with involution. We say that a homomorphism of algebras $\varphi: \cA \rightarrow \cB$ is an \textit{involution preserving homomorphism} if it commutes with the involution, i.e., $\varphi \circ -_{\cA} = -_{\cB} \circ \varphi$. If there is no confusion, we will denote involutions by ``$-$".
\end{df}

Assume now that $\FF$ is a field of characteristic different from 2, 3 and 5. We will only consider finite-dimensional algebras. Smirnov proved in \cite[Theorem 2.1]{Smi90b} that any semisimple structurable algebra is the direct sum of simple algebras. The simple algebras are central simple over their center, and thus the description of semisimple algebras is reduced to the description of central simple algebras. 

\begin{theorem}\label{th:classification.structurable}
(\cite[Theorem 3.8]{Smi90b}, see also \cite[Theorem 11]{All79})
Any central simple structurable algebra is isomorphic to one of the following:

\noindent
(a) a Jordan algebra (with the identity involution),

\noindent
(b) an associative algebra with involution,

\noindent
(c) a $2 \times 2$ matrix algebra constructed from the Jordan algebra $\mathcal{J}$ of an admissible cubic form with basepoint and a nonzero scalar, 

\noindent
(d) an algebra with involution constructed from an hermitian form,

\noindent
(e) a tensor product $(\cC \otimes \cB ,-)$ where $\cC$ is a Cayley algebra, $\cB$ is a Hurwitz algebra and $-$ is the tensor product of the standard involutions, or a form of such a tensor product algebra,

\noindent
(f) a Kantor-Smirnov central simple algebra $\smir$ with involution constructed from an octonion algebra.
\end{theorem}

%%%%%%%%%%%%%%%%%%%%%%%%%%%%%%%%%%%%%%%%%%%%%%%%%%%%%%%%%%%%%%%%%%%%%%%%%%%%%%%%%%%%%%
%%%%%%%%%%%%%%%%%     *SUBSECTION*   HURWITZ ALGEBRAS             %%%%%%%%%%%%%%%%%%%%
%%%%%%%%%%%%%%%%%%%%%%%%%%%%%%%%%%%%%%%%%%%%%%%%%%%%%%%%%%%%%%%%%%%%%%%%%%%%%%%%%%%%%%
\subsection{Hurwitz algebras}\label{subsection.Hurwitz}

Hurwitz algebras constitute a generalization of the classical algebras of the real $\R$,
complex $\mathbb{C}$, quaternion $\mathbb{H}$ and octonion numbers $\mathbb{O}$.

\begin{df}\label{df:4.1}
A \textit{composition algebra} over a field $\FF$ is a not necessarily associative algebra $\cB$, endowed with a nondegenerate quadratic form (the \textit{norm}) $n: \cB \rightarrow \FF$ (i.e., the bilinear form $n(x,y) := n(x+y)-n(x)-n(y)$ is nondegenerate) which is multiplicative: $n(xy)=n(x)n(y)$ for all $x,y \in \cB$. The unital composition algebras are called \textit{Hurwitz algebras}.
\end{df}

Hurwitz algebras of dimension 4 and 8 are called, respectively, \textit{quaternion} and \textit{Cayley} (or \textit{octonion}) algebras.

\begin{df}\label{df:antisymmetric.elements}
The map $x \mapsto \overline{x}:= n(x,1)1 -x$ is an involution of the Hurwitz algebra $\cB$ called the \textit{standard conjugation}. We will denote $\cS(\cB,-)$ by $\cB_0$. 
\end{df}

The algebra obtained from a subalgebra $Q$ of a Hurwitz algebra through the Cayley-Dickson doubling process is denoted by $CD(Q, \alpha)$ where $0 \neq \alpha \in \FF$ (see \cite[p. 125]{EKmon}). 

Every Hurwitz algebra (recall that $\ch \FF \neq 2$) is isomorphic either to the ground field $\FF$, a quadratic algebra $CD(\FF, \alpha)$, a quaternion algebra $CD(\FF, \alpha,\beta)$ or a Cayley algebra  $CD(\FF, \alpha,\beta,\gamma)$ for $\alpha, \beta,\gamma \in \FF$ (\cite[Corollary 4.6]{EKmon}).

\textit{Suppose now that $\FF$ is algebraically closed}, then the norm $n$ is \textit{isotropic}, i.e., there exist nonzero elements of norm $0$. It is well-known that in this case there is only one Hurwitz algebra, up to isomorphism, for each possible dimension $1$, $2$, $4$ and $8$.

Up to isomorphism, the unique Cayley algebra $\cC$ is called the \textit{split Cayley algebra}. A well-known basis that is usually called a \textit{canonical basis} or \textit{good basis} of $\cC$ is $\lbrace e_1, e_2, u_1, u_2, u_3, v_1, v_2, v_3 \rbrace$, with multiplication table (see \cite[Figure~4.1]{EKmon}):

\begin{center}
\begin{tabular}{c||cc|ccc|ccc|}
 & $e_1$ & $e_2$ & $u_1$ & $u_2$ & $u_3$ & $v_1$ & $v_2$ & $v_3$ \\
 \hline\hline
 $e_1$ & $e_1$ & 0 & $u_1$ & $u_2$ & $u_3$ & 0 & 0 & 0 \\
 $e_2$ & 0 & $e_2$ & 0 & 0 & 0 & $v_1$ & $v_2$ & $v_3$ \\
 \hline
 $u_1$ & 0 & $u_1$ & 0 & $v_3$ & $-v_2$ & $-e_1$ & 0 & 0 \\
 $u_2$ & 0 & $u_2$ & $-v_3$ & 0 & $v_1$ & 0 & $-e_1$ & 0 \\
 $u_3$ & 0 & $u_3$ & $v_2$ & $-v_1$ & 0 & 0 & 0 & $-e_1$ \\
 \hline
 $v_1$ & $v_1$ & 0 & $-e_2$ & 0 & 0 & 0 & $u_3$ & $-u_2$ \\
 $v_2$ & $v_2$ & 0 & 0 & $-e_2$ & 0 & $-u_3$ & 0 & $u_1$ \\
 $v_3$ & $v_3$ & 0 & 0 & 0 & $-e_2$ & $u_2$ & $-u_1$ & 0 \\
 \hline
\end{tabular}
\end{center}

The subalgebra $\cK=\FF e_1 +\FF e_2$ (resp. $\cH = \FF e_1 +\FF e_2 + \FF u_1 +\FF v_1$) is, up to isomorphism, the unique Hurwitz algebra in dimension 2 (resp. 4). $\cK$ is called the \textit{split quadratic algebra} and $\cH$ the \textit{split quaternion algebra} (see \cite[Theorem 4.8]{EKmon}). 

\textit{The grading groups will be considered to be abelian, unless otherwise stated}.  This is a choice of the authors motivated by the fact that when dealing with group gradings on Hurwitz algebras, it is enough to restrict ourselves to gradings by abelian groups (see \cite[Proposition 4.10]{EKmon}). 
The term \textit{grading} will refer to grading by abelian group and the term \textit{universal group} will refer to abelian universal group (see \cite[Remark 1.22]{EKmon}).

The following two gradings are the only fine gradings, up to equivalence, on the Cayley algebra $\cC$ (see \cite[Corollary 4.14]{EKmon}):

\begin{itemize}

\item The $\Z^2$-grading with homogeneous components given by (considering the canonical basis)
\begin{equation}\label{eq:cartan.grading}
\begin{array}{ll}
\cC_{(0,0)}= \FF e_1 \oplus \FF e_2, & \\
\cC_{(1,0)}= \FF u_1, & \cC_{(-1,0)}= \FF v_1,\\
\cC_{(0,1)}= \FF u_2, & \cC_{(0,-1)}= \FF v_2,\\
\cC_{(1,1)}= \FF v_3, & \cC_{(-1,-1)}= \FF u_3.
\end{array}
\end{equation}
This grading is called the \textit{Cartan grading} and its universal group is $\Z^2$. We denote this grading by \textbf{$\Gamma^1_{\cC}$}.

\item The $(\Z/2)^3$-\textit{grading induced by the Cayley-Dickson doubling process} with homogeneous components given by (considering a basis associated to the Cayley-Dickson doubling process $\lbrace 1,w,v,vw,u,uw,vu,(wv)u\rbrace$)
\begin{equation}\label{eq:doubling.process.grading}
\begin{array}{ll}
\cC_{(\bar{0},\bar{0},\bar{0})}=\FF 1, & \cC_{(\bar{1},\bar{1},\bar{0})}=\FF uv,\\
\cC_{(\bar{1},\bar{0},\bar{0})}=\FF u, & \cC_{(\bar{1},\bar{0},\bar{1})}=\FF uw,\\
\cC_{(\bar{0},\bar{1},\bar{0})}=\FF v, & \cC_{(\bar{0},\bar{1},\bar{1})}=\FF vw,\\
\cC_{(\bar{0},\bar{0},\bar{1})}=\FF w, & \cC_{(\bar{1},\bar{1},\bar{1})}=\FF (uv)w.
\end{array}
\end{equation}
The universal group of this grading is $(\Z/2)^3$. (In terms of the good basis, we can take, for instance, $u=e_1-e_2$, $v=u_1+v_1$, $w=u_2+v_2$.) We denote this grading by \textbf{$\Gamma^2_{\cC}$}.
\end{itemize}

\begin{remark}\label{re:fine_gradings_on_dim=2,4}
The fine gradings, up to equivalence, on a quaternion algebra $\cH$ are the following (see \cite[Remark 4.16]{EKmon}):
\begin{itemize}
\item The Cartan grading over its universal group $\Z$. In this case $\cH$ has a basis $\lbrace e_1,e_2,u_1,v_1 \rbrace$ with multiplication table
\begin{center}
\begin{tabular}{ c | c c c c}
 & $e_1$ & $e_2$ & $u_1$ & $v_1$ \\
\hline
$e_1$ & $e_1$ & $0$ & $u_1$ & $0$ \\
$e_2$ & $0$ & $e_2$ & $0$ & $v_1$ \\
$u_1$ & $0$ & $u_1$ & $0$ & $-e_1$ \\
$v_1$ & $v_1$ & $0$ & $-e_2$ & $0$
\end{tabular}
\end{center}
and the homogeneous components are given by
\begin{equation}\label{eq:cartan.grading.quaternions}
\cH_0= \FF e_1 \oplus \FF e_2, \quad \cH_1= \FF u_1, \quad \cH_{-1}= \FF v_1.
\end{equation}

\item The $(\Z/2)^2$-grading induced by the Cayley-Dickson doubling process. Considering a basis associated to the Cayley-Dickson doubling process $\lbrace 1, u, v, uv \rbrace$, the homogeneous components are given by 
\begin{equation}\label{eq:C-D.grading.quaternions}
\cH_{(\bar{0},\bar{0})}= \FF 1, \quad \cH_{(\bar{1},\bar{0})}=\FF u, \quad \cH_{(\bar{0},\bar{1})}= \FF v, \quad \cH_{(\bar{1},\bar{1})}= \FF uv.
\end{equation}
(In terms of the good basis, we can take, for instance, $u=e_1-e_2$, $v=u_1+v_1$.)
\end{itemize}
And the only nontrivial grading on a Hurwitz algebra $\cK$ of dimension 2, up to equivalence, is the one induced by the Cayley-Dickson doubling process by $\Z/2$. The homogeneous components are given by (considering a basis associated to the Cayley-Dickson doubling process $\lbrace 1,u\rbrace$)
\begin{equation}\label{eq:C-D.grading.dim2}
\cK_{\bar{0}}=\FF 1, \quad \cK_{\bar{1}}=\FF u.
\end{equation}
\end{remark}

\begin{remark}\label{re:C_0}
Let $\cC$ be the Cayley algebra.
\begin{enumerate} 
\item Consider a basis $\lbrace 1,u,v,w,uv,uw,vw,(uv)w\rbrace$ of $\cC$ given by the Cayley-Dickson doubling process. We have that $\lbrace u,v,w,uv,uw,vw,(uv)w\rbrace$ is a basis for the subspace $\cC_0$.

\item Consider the good basis $\{ e_1,e_2,u_1,u_2,u_3,v_1,$ $v_2,v_3 \}$. We have that $\{ e_1-e_2,u_1,u_2,u_3, v_1,v_2,v_3 \}$ is a basis for the subspace $\cC_0$.
\end{enumerate}
$\cC_0$ is an algebra with the multiplication given by the commutator.
Observe that the subspace of skew elements generates the whole Cayley algebra if we consider the usual multiplication, therefore it is enough to know the degree of the homogeneous elements of $\cC_0$ to determine the grading on $\cC$. It is easy to see that $\dg 1=e$ where $e$ is the neutral element of the group.
\end{remark}

%%%%%%%%%%%%%%%%%%%%%%%%%%%%%%%%%%%%%%%%

\subsection{The Smirnov algebra}

We will refer to the algebras in class (f) of Theorem~\ref{th:classification.structurable} as Smirnov algebras.

Smirnov algebras are $35$-dimensional simple exceptional structurable algebras. It is well-known (\cite{Smi90a}) that its derivation algebra is a simple Lie algebra of type $G_2$, and its Kantor construction is a simple Lie algebra of type $E_7$. We will recall now the definition of the Smirnov algebra.

\smallskip

In this construction, we always assume that $\cC$ is a Cayley algebra over a field $\FF$ of characteristic different from $2$, with norm $n$ and product $\cdot$\hspace*{0.05cm}; the bilinear form associated to the norm will be denoted by $n$ too. We recall now the construction of the Smirnov algebra (see \cite{Smi90a} for more details). Denote by $\cS$ the $7$-dimensional subspace of skew-symmetric elements of $\cC$ and let $[\cdot,\cdot]$ be the commutator in $\cC$. Then, $(\cS, [\cdot,\cdot])$ is a central non-Lie Malcev algebra, which is denoted by $\cS^{(-)}$. It is well-known that there is a nondegenerate symmetric bilinear form $f$ on $\cS$ satisfying
\begin{align}
& [[x,y],y] = f(y,y)x - f(x,y)y, \label{eqf} \\
& f(x, [z,y]) = f([x,z], y),
\end{align}
for any $x,y,z \in \cS$. A straightforward computation shows that
\begin{equation} \label{eqfn}
f(x,y) = -2 n(x,y)
\end{equation}
for any $x,y\in\cS$. (Although the product of the Smirnov algebra was defined in \cite{Smi90a} using the form $f$, we will use $n$ instead, as in \cite{AF93b}.) Let $M$ denote the subpace of $\cS \otimes \cS$ generated by the set $\{ s_1 \otimes s_2 - s_2 \otimes s_1 \med s_1,s_2\in\cS\}$, and set $\cH = (\cS \otimes \cS)/M$. We write $s_1 \times s_2 = s_1 \otimes s_2 + M$ for $s_1,s_2\in\cS$. On $\cH \oplus \cS$ we define a commutative product $\odot$ and an anticommutative product $[\cdot,\cdot]$ given by
\begin{equation} \label{product} \begin{aligned}
& s_1 \odot s_2 = s_1 \times s_2, \\
& s \odot (s_1 \times s_2) = -n(s_1,s_2)s - \frac{1}{2} \Big( n(s,s_1)s_2 + n(s,s_2)s_1 \Big), \\
& (s_1 \times s_2) \odot (s_3 \times s_4) = \frac{1}{4} \Big( [s_1,s_3]\times[s_2,s_4] + [s_1,s_4]\times[s_2,s_3] \Big) - \\
& \qquad\qquad\qquad\qquad -n(s_1,s_2) s_3 \times s_4 -n(s_3,s_4) s_1 \times s_2 \\
& [s_1,s_2] = [s_1,s_2], \\
& [s, s_1\times s_2] = [s,s_1] \times s_2 + s_1 \times [s,s_2], \\
& [s_1 \times s_2, s_3 \times s_4] = -\frac{1}{2} \Big( n(s_1,s_3)[s_2,s_4] + n(s_1,s_4)[s_2,s_3] + \\
& \qquad\qquad\qquad\qquad + n(s_2,s_3)[s_1,s_4] + n(s_2,s_4)[s_1,s_3] \Big),
\end{aligned} \end{equation}
for any $s, s_i\in\cS$ and where the brackets $[\cdot,\cdot]$ on the right side of the equalities denote the product of $\cS^{(-)}$. 
(Note that the third equation in \eqref{product} has a coefficient wrong in \cite{Smi90a}, which is corrected in \cite{Smi92}.)
Then, the vector space $\smir \coloneqq \cH \oplus \cS$ with the new product
\begin{equation}
xy = x \odot y + \frac{1}{2}[x,y]
\end{equation}
and the involution given by $h+s \mapsto h-s$, for $h\in\cH$ and $s\in\cS$, define a $35$-dimensional simple structurable algebra that is called the \emph{Smirnov algebra} (or \emph{Kantor-Smirnov algebra}) associated to $\cC$.
It is easy to see that $\cH$ and $\cS$ are the subspaces of symmetric and skew-symmetric elements, respectively, and we have $x \odot y = \frac{1}{2}(xy+yx)$ and $[x,y] = xy-yx$.

\bigskip

In \cite{AF93b}, Allison and Faulkner proved that $\smir$ is isomorphic to a subalgebra of the structurable algebra $\CxC$, and gave a construction of $\smir$ different from Smirnov's construction. This second construction, which we will denote by $\smirCxC$ to avoid confusion, is given by
\begin{equation}
\smirCxC := \lspan\{ a \otimes a - n(a)1 \otimes 1 \med a\in\cC\},
\end{equation}
where the involution is the restriction of the involution of $\CxC$, and
\begin{equation}\begin{aligned}
\cS(\smirCxC) &= \{s \otimes 1 + 1 \otimes s \med s\in\cS \}, \\
\cH(\smirCxC) &= \{s \otimes s - n(s) 1 \otimes 1 \med s\in\cS \}.
\end{aligned}\end{equation}
It is known that if $\{e_i\}_{i=1}^7$ is an orthogonal basis of $\cS$ with respect to $f$ (or $n$) and $f(e_i,e_i)=\alpha_i$ for $i=1,\dots, 7$, then the identity of $\smirCxC$ can be written as $1_\smirCxC=\sum_{i=1}^7\frac{1}{4 \alpha_i} e_i \times e_i$. In other words, if $\{x_i\}_{i=1}^7$ is an orthogonal basis of $\cS$ with respect to $n$ and $n(x_i) = \alpha_i$ for $i=1,\dots, 7$, then $1_\smirCxC = \sum_{i=1}^7\frac{-1}{16 \alpha_i} x_i \times x_i$.

\smallskip

An isomorphism $\psi\colon \smir \to \smirCxC$ between the two constructions (see \cite[Proof of Prop.~1.9]{AF93b}) is given by
\begin{equation}\label{isomSmirnovConstructions}
\psi(s) = s \otimes 1 + 1 \otimes s, \quad \psi(s \times s) = 2(s \otimes s - n(s) 1 \otimes 1)
\end{equation}
for $s\in\cS$.

\smallskip

\begin{df}
The linear form determined by
\begin{align*}
\tr \colon \smir & \longrightarrow \FF \\
(s_1 \times s_2) + s & \longmapsto -8 n(s_1,s_2)
\end{align*}
for $s_1,s_2,s\in\cS$ will be called the \emph{(linear) trace} of the Smirnov algebra $\smir$. We also denote by $\tr$ the associated bilinear form $\tr \colon \smir \times \smir \to \FF$ given by $\tr(x, y) \coloneqq \tr(x \bar y)$ for $x,y \in \smir$, that will be called the \emph{(bilinear) trace} of $\smir$. By abuse of notation and when there is no confusion, we will refer to any of these trace forms as the \emph{trace} of $\smir$. Since $\tr(x) = \tr(x, 1)$ for $x,y\in\smir$, each trace form determines the other one. Also, note that $\tr(1) = 7$ coincides with the degree of the norm of $\smir$.
\end{df}

%%%%%%%%%%%%%%%%%%%%%%%%%%%%%%%%%%%%%%%%%%%%%%%%%%%%%%%%%%%%%%%%%%%%%%%%%%%%%%%%%%%%%%
%%%%%%%%%%%      *SECTION*   TENSOR PRODUCT OF COMPOSITION ALGEBRAS        %%%%%%%%%%%
%%%%%%%%%%%%%%%%%%%%%%%%%%%%%%%%%%%%%%%%%%%%%%%%%%%%%%%%%%%%%%%%%%%%%%%%%%%%%%%%%%%%%%
\section{Tensor product of composition algebras}\label{section.Tensor}

In this section we study involution preserving gradings on $\cC \otimes \cB$ where $\cC$ is a Cayley algebra and $\cB$ is a Hurwitz algebra. We start by proving in Section \ref{subsection.Hurwitz.Automorphism.scheme.Cayley} that $\AAut (\cC^1 \otimes \cdots \otimes \cC^n, -\otimes \cdots \otimes -) \simeq \AAut(\cC^1 \times \cdots \times \cC^n)$ for Cayley algebras $\cC^1,...,\cC^n$. This shows that there is a correspondence between gradings on $\cC^1 \times \cC^2$ and involution preserving gradings on $\cC^1 \otimes \cC^2$.
In Section \ref{subsection.Tensor.gradings.on.tensor.of.Hurwitz} we give the classification of involution preserving gradings, up to equivalence and isomorphism, on the tensor product of a Cayley algebra and a Hurwitz algebra of dimension 2 and 4. 
In Section \ref{subsection.Tensor.direct.product} we give involution preserving gradings, up to equivalence and isomorphism, on $\cC^1 \times \cC^2$ and finally in Section \ref{subsection.Tensor.tensor.product} we give involution preserving gradings, up to equivalence and isomorphism, on $\cC^1 \otimes \cC^2$. (Recall that, as before, all grading groups considered are assumed to be abelian.)
As we mentioned before the tensor product of a Cayley algebra $(\cC,-)$ and the field $\FF$ is isomorphic (as algebras with involution) to $(\cC,-)$ and since gradings on Cayley algebras are already known (see Section \ref{subsection.Hurwitz}), we omit this case.

%%%%

\subsection[Tensor product of Cayley algebras]{Automorphism scheme of the tensor product of Cayley algebras}\label{subsection.Hurwitz.Automorphism.scheme.Cayley}

In this section we use definitions and results from \cite{MPP01} to prove that
\[
\AAut (\cC^1 \otimes \cdots \otimes \cC^n) \simeq \AAut(\cC^1 \times \cdots \times \cC^n)
\]
where $\cC^i$ are Cayley algebras for $i=1,...,n$. This reduces the problem of classifying gradings on $\cC^1 \otimes \cdots \otimes \cC^n$ to classify gradings on $\cC^1 \times \cdots \times \cC^n$. 

\begin{df}\label{df:3.1.mpp}\cite[Definition 3.1]{MPP01}
The \textit{generalized alternative nucleus} of an algebra $\cA$ is defined by
\[
N_{alt} (\cA) := \lbrace a \in \cA : (a,x,y)= - (x,a,y)= (x,y,a) \hspace{0.2cm} \forall x,y \in \cA \rbrace,
\]
where $(a,x,y):= (ax)y-a(xy)$ for all $a,x,y \in \cA$.

\end{df}

\begin{remark}\label{re:deriv.algebra.N}
Let $\cC^i$ be the Cayley algebra for $i=1,...,n$. Recall that $\cC^i_0$ is the subspace of skew elements of $\cC^i$. Identify $\cC^i$ with $1\otimes\cdots \otimes\cC^i\otimes\cdots \otimes 1$ for $i=1,...,n$. We find in \cite{MPP01} that 
\[
N_{alt} (\cC^1 \otimes \cdots \otimes \cC^n) = \FF 1 \oplus \cC^1_0 \oplus \cdots \oplus \cC^n_0 = \cC^1 +\cdots+ \cC^n.
\]
And the \textit{derived algebra} of $N_{alt}(\cC^1 \otimes \cdots \otimes \cC^n)$ is
\[
N'_{alt}(\cC^1 \otimes \cdots \otimes \cC^n) = [N_{alt}(\cC^1 \otimes \cdots \otimes \cC^n), N_{alt}(\cC^1 \otimes \cdots \otimes \cC^n)]= \cC^1_0 \oplus \cdots \oplus \cC^n_0.
\]
\end{remark}

\begin{remark}\label{re:deriv.lie}
Let $\mathbf{G}$ be an affine algebraic group scheme and let $\mathcal{A}$ be an algebra. In \cite[p. 316 and 313]{EKmon} we find the following statements:

\begin{enumerate}
\item[i)] $\dm \Lie (\mathbf{G})\geq \dm \mathbf{G} = \dm \mathbf{G}(\overline{\FF})$.

\item[ii)] $\Lie (\AAut (\mathcal{A}))= \Der (\mathcal{A})$.

\item[iii)] $\AAut (\mathcal{A})$ is smooth if and only if $\dm \Der (\mathcal{A})= \dm \Aut_{\overline{\F}}
(\mathcal{A} \otimes \overline{\F})$.
\end{enumerate}
\end{remark}

From now on we will use the identification $\cA_1 \times \cdots \times \cA_n \simeq \cA_1 \oplus \cdots \oplus \cA_n $.

\begin{remark}\label{re:smooth.subscheme}
Let $\cC,\cC^1,...,\cC^n$ be Cayley algebras.
\begin{enumerate}
\item[i)] $\AAut (\cC)$ is smooth (\cite[p. 146]{EKmon}).

\item[ii)] By \cite[Proposition 3.6]{MPP01} we have that the restriction map $\AAut (\cC) \rightarrow \AAut (\cC_0)$ satisfies conditions 1) and 2) of Theorem \ref{th:A.50} and by i) we have that $\AAut (\cC) \simeq \AAut (\cC_0)$.

\item[iii)] We claim that $\AAut (\cC^1_0) \times \cdots \times \AAut (\cC^n_0)$ is subgroupscheme of $\AAut (\cC^1_0 \times \cdots \times \cC^n_0)$. Let $R$ be an object in $\Alg_{\FF}$. Consider the isomorphism
\[
\begin{array}{llll}
h: & (\cC^1_0 \times \cdots \times \cC^n_0) \otimes R &
\longrightarrow & (\cC^1_0 \otimes R) \times \cdots \times (\cC^n_0 \otimes R)\\
&(c_1,...,c_n)\otimes r & \longmapsto & (c_1\otimes r, ..., c_n\otimes r)
\end{array}
\]
for $c_i \in \cC^i_0$ and $r \in R$, $i=1,...,n$, and the canonical inclusion
\[
\begin{array}{ccc}
\iota : \Aut_R(\cC^1_0 \otimes R) \times \cdots \times \Aut_R(\cC^n_0 \otimes R) &
\longrightarrow & \Aut_R((\cC^1_0 \otimes R) \times \cdots \times (\cC^n_0 \otimes R)) \\
(f_1,\dots,f_n) & \longmapsto & f_1\times\cdots\times f_n,
\end{array}
\]
for $f_i \in \Aut_R (\cC^i_0 \otimes R)$. Then, for each $R$ we can define a monomorphism
\[
\begin{array}{ccc}
\Aut_R(\cC^1_0 \otimes R) \times \cdots \times \Aut_R(\cC^n_0 \otimes R) & 
\longrightarrow & \Aut_R((\cC^1_0 \times \cdots \times \cC^n_0) \otimes R)\\
(f_1,\dots,f_n) & \longmapsto & h^{-1}\circ \iota(f_1,\dots,f_n) \circ h
\end{array}
\]
and these behave well with morphisms, which proves the claim.
\item[iv)] Let $\mathbf{G}$ and $\mathbf{H}$ be affine group schemes. We can define $\mathbf{G} \times \mathbf{H}$ whose representing object is $\FF[\mathbf{G}]\otimes\FF[\mathbf{H}]$ (see \cite[p. 300]{EKmon}). As a consequence of Noether's normalization Lemma (\cite[Chapter 8, Section 13]{Jac85}) we get $\dm(\mathbf{G} \times \mathbf{H})= \dm \mathbf{G}+\dm \mathbf{H}$. 
\end{enumerate}
\end{remark}

\begin{lemma}\label{le:grad_involution_tensor_Cayley}
Let $\cC^1,...,\cC^n$ be Cayley algebras and let $\sigma= -\otimes \cdots \otimes -$ be the involution of $\cC^1\otimes \cdots \otimes\cC^n$, i.e., the tensor product of the involutions of each $\cC^i$ for $i=1,...,n$. Then
\[
\AAut(\cC^1\otimes \cdots \otimes\cC^n)= \AAut(\cC^1\otimes \cdots \otimes\cC^n, \sigma)
\]
where, $\AAut(\cC^1\otimes \cdots \otimes\cC^n, \sigma)(R)= \{ \varphi \in \Aut_{R-\mathrm{alg}} (\cC^1\otimes \cdots \otimes\cC^n \otimes R): \varphi \circ (\sigma \otimes id_R) = (\sigma \otimes id_R)\circ \varphi \}$ for $R\in \Alg_{\FF}$ (see Definition \ref{df:structurable_grading}).
\end{lemma}

\begin{proof}
Let $R$ be an arbitrary object in $\Alg_{\FF}$. We will prove that
\[
\Aut_{R-\mathrm{alg}}(\cC^1\otimes \cdots \otimes\cC^n\otimes R) = \Aut_{R-\mathrm{alg}}(\cC^1\otimes \cdots \otimes\cC^n\otimes R, \sigma).
\]
$\supseteq$ is trivial. By Remark \ref{re:deriv.algebra.N}
$N_{alt}(\cC^1\otimes \cdots \otimes\cC^n\otimes R)= (\cC^1+\cdots+\cC^n)\otimes R.$

 Then
\[
\begin{array}{l}
[N_{alt}(\cC^1\otimes \cdots \otimes\cC^n\otimes R), N_{alt}(\cC^1\otimes \cdots \otimes\cC^n\otimes R)] \\
\hspace*{1cm}= [(\cC^1+\cdots+\cC^n)\otimes R, (\cC^1+\cdots+\cC^n)\otimes R] \\
\hspace*{1cm}= (\cC^1_0\oplus \cdots\oplus\cC^n_0)\otimes R,
\end{array}
\]
which generates $\cC^1\otimes \cdots \otimes\cC^n\otimes R$ (see Remark \ref{re:C_0}). Consider $\varphi\in \Aut_{R-\mathrm{alg}}(\cC^1\otimes \cdots \otimes\cC^n\otimes R)$. Notice that $[N_{alt}(\cC^1\otimes \cdots \otimes\cC^n\otimes R), N_{alt}(\cC^1\otimes \cdots \otimes\cC^n\otimes R)]$ is invariant under $\varphi$. For $x_i\in \cC^i_0$ and $r_i\in R$, $i=1,...,n$ we have
\[
\begin{array}{l}
\sigma\otimes id_R (x_1\otimes 1\otimes\cdots\otimes r_1 +\cdots+ 1\otimes\cdots\otimes x_n\otimes r_n)\\
\hspace*{1.1cm} = \bar{x}_1\otimes 1\otimes\cdots\otimes r_1 +\cdots+ 1\otimes\cdots\otimes \bar{x}_n\otimes r_n\\
\hspace*{1.1cm}=-(x_1\otimes 1\otimes\cdots\otimes r_1 +\cdots+ 1\otimes\cdots\otimes x_n\otimes r_n),
\end{array}
\]
then $\sigma\otimes id_R = -id_{(\cC^1_0\oplus\cdots\oplus\cC^n_0)\otimes R}$ in $(\cC^1_0\oplus\cdots\oplus\cC^n_0)\otimes R$. Hence
\[
\varphi \circ (\sigma\otimes id_R) = (\sigma\otimes id_R)\circ \varphi
\]
in $(\cC^1_0\oplus\cdots\oplus\cC^n_0)\otimes R$. Therefore $\varphi \circ (\sigma\otimes id_R) = (\sigma\otimes id_R)\circ \varphi$ in $\cC^1\otimes \cdots \otimes\cC^n\otimes R$, so $\varphi\in \Aut_{R-\mathrm{alg}}(\cC^1\otimes \cdots \otimes\cC^n\otimes R,\sigma)$.

\end{proof}

Using the above results, we have the following:

\begin{theorem}\label{th:iso.aut.schemes.cayley}
Let $\cC^i$ be the Cayley algebra for $i=1,...,n$.
Then there exist isomorphisms of schemes $\Phi$ and $\varphi$
\[
\AAut (\cC^1 \otimes \cdots \otimes \cC^n) \stackrel{\Phi}{\rightarrow}
\AAut (\cC^1_0 \times \cdots \times \cC^n_0) \stackrel{\varphi}{\leftarrow}
\AAut (\cC^1 \times \cdots \times \cC^n)
\]
where $\Phi(R)(f) = f|_{N'_{alt}(\cC^1 \otimes \cdots \otimes \cC^n)}$ and $\varphi(R)(g)=
g |_{[\cC^1 \times \cdots \times \cC^n, \cC^1 \times \cdots \times \cC^n]}$ for
$f \in \AAut (\cC^1 \otimes \cdots \otimes \cC^n)(R)$, $g \in \AAut (\cC^1 \times \cdots \times \cC^n)(R)$ and an object $R$ in $\Alg_{\FF}$.

Moreover,
\[
\AAut(\cC^1\otimes \cdots \otimes\cC^n)= \AAut(\cC^1\otimes \cdots \otimes\cC^n, \sigma),
\]
where $\sigma$ is the involution of $\cC^1\otimes \cdots \otimes\cC^n$.

\end{theorem}

\begin{proof}
By Theorem \ref{th:A.50} we see that in order to prove that $\Phi$ is an isomorphism of schemes it is enough to show that $\Phi(\overline{\FF})$ and $d \Phi$ are bijective and $\AAut (\cC^1_0 \times \cdots \times \cC^n_0)$ is smooth.

By \cite[Proposition 3.6]{MPP01} and Remark \ref{re:deriv.algebra.N} we get that $\Phi({\overline{\FF}})$ is bijective. By \cite[Proposition 3.6]{MPP01} and Remark \ref{re:deriv.lie} ii) we get that $d \Phi$ is bijective.

To prove that $\AAut (\cC^1_0 \times \cdots \times \cC^n_0)$ is smooth, by Remark \ref{re:deriv.lie} iii), it is enough to show that $\dm \Der (\cC^1_0 \times \cdots \times \cC^n_0)= \dm \Aut_{\overline{\FF}}((\cC^1_0 \times \cdots \times \cC^n_0) \otimes {\overline{\FF}})$.
We have
\[
\begin{split}
\sum_{i=1}^{n} \dm \AAut (\cC^i)& = \sum_{i=1}^{n} \dm \AAut (\cC^i_0) \hspace{0.2cm} \mbox{ (by
Remark \ref{re:smooth.subscheme} ii))}\\
& \leq \dm \AAut (\cC^1_0 \times \cdots \times \cC^n_0) \hspace{0.2cm} \mbox{ (by Remark
\ref{re:smooth.subscheme} iii) and iv))}\\
& \leq \dm \Lie (\AAut (\cC^1_0 \times \cdots \times \cC^n_0)) \hspace{0.2cm} \mbox{ (by Remark
\ref{re:deriv.lie} i))}\\
& = \dm \Der (\cC^1_0 \times \cdots \times \cC^n_0) \hspace{0.2cm} \mbox{ (by Remark
\ref{re:deriv.lie} ii))}\\
& = \sum_{i=1}^{n} \dm \Der (\cC^i) \hspace{0.2cm} \mbox{ (by \cite[Proposition 3.6]{MPP01})}\\
& = \sum_{i=1}^{n} \dm \AAut (\cC^i) \hspace{0.2cm} \mbox{ (by Remark \ref{re:smooth.subscheme} i)
and \ref{re:deriv.lie} iii))}.
\end{split}
\]
Then $\dm \AAut (\cC^1_0 \times \cdots \times \cC^n_0) = \dm \Lie (\AAut (\cC^1_0 \times \cdots \times \cC^n_0))$, from this and $\dm \AAut (\cC^1_0 \times \cdots \times \cC^n_0) = \dm \AAut (\cC^1_0 \times \cdots \times \cC^n_0) (\overline{\FF})$ (Remark \ref{re:deriv.lie} i)) follows:
\[
\begin{split}
\dm \AAut (\cC^1_0 \times \cdots \times \cC^n_0) (\overline{\FF}) & =
\dm \Lie (\AAut (\cC^1_0 \times \cdots \times \cC^n_0))\\
&= \dm \Der (\cC^1_0 \times \cdots \times \cC^n_0) \mbox{ (Remark
\ref{re:deriv.lie} ii))}
\end{split}
\]
Therefore $\AAut (\cC^1_0 \times \cdots \times \cC^n_0)$ is smooth and then $\Phi$ is an
isomorphism of schemes.

In order to prove that $\varphi$ is an isomorphism of schemes we will use again Theorem \ref{th:A.50}. We already proved that $\AAut (\cC^1_0 \times \cdots \times \cC^n_0)$ is
smooth, so we only have to prove that $\varphi (\overline{\FF})$ and $d \varphi$ are bijective.
Take $g \in \AAut (\cC^1 \times \cdots \times \cC^n) (\overline{\FF})$, then $g$ permutes the minimal ideals of $\cC^1 \times \cdots \times \cC^n$ which are $0 \times \cdots \times \cC^i \times \cdots \times 0$ for $i=1,...,n$.
Then there exists $\alpha \in S_n$ such that $g (0 \times \cdots \times \cC^i \times \cdots \times 0) = 0 \times \cdots \times \cC^{\alpha(i)} \times \cdots \times 0$
for $i=1,...,n$. Consider the isomorphism
\[
\begin{array}{rlll}
g_i: & \cC^i & \longrightarrow & \cC^{\alpha(i)} \\
& x & \longmapsto & P_{\alpha(i)}(g(0,...,\stackbin{i}{x},...,0))
\end{array}
\]
where $P_{\alpha(i)}$ is the canonical projection in the ${\alpha(i)}$-th entry. We have
\[
\begin{array}{llllll}
g = & \cC^1 \times \cdots \times \cC^n & \longrightarrow & \cC^{\alpha(1)} \times \cdots \times \cC^{\alpha(n)} & \longrightarrow &  \cC^1 \times \cdots \times \cC^n\\
 & (c_1,...,c_n) & \longmapsto & (g_1 (c_1),..., g_n(c_n)) & & \\
 &  &  & (a_1,...,a_n) & \longmapsto &  (a_{\alpha^{-1}(1)},..., a_{\alpha^{-1}(n)}).
\end{array}
\]
Since $\cC_0^i=[\cC^i,\cC^i]$, $g_i |_{\cC_0^i}: \cC_0^i \longrightarrow \cC_0^{\alpha(i)}$ is an isomorphism. Then
\[
g (\cC^1_0 \times \cdots \times \cC^n_0) = \cC^1_0 \times \cdots \times \cC^n_0,
\]
i.e., $\cC^1_0 \times \cdots \times \cC^n_0$ is invariant under $g$ and since $\cC^1_0 \times \cdots \times \cC^n_0$ generates $\cC^1 \times \cdots \times \cC^n$, we have that $\varphi(\overline{\FF})$ is injective.

Take $f \in \AAut (\cC^1_0 \times \cdots \times \cC^n_0)(\overline{\FF}) = \Aut (\cC^1_0 \times \cdots \times \cC^n_0)$. Notice that $0 \times \cdots \times \cC^i_0 \times \cdots \times 0$ is a minimal ideal of $\cC^1_0 \times \cdots \times \cC^n_0$ for $i=1,...,n$. Then there exists $\tau \in S_n$ such that $f(0 \times \cdots \times \cC^i_0 \times \cdots \times 0) = 0 \times \cdots \times \cC^{\tau(i)}_0 \times \cdots \times 0$.
Then $f$ induces the isomorphisms for $i=1,...,n$
\[
\begin{array}{llll}
f_i: & \cC^i_0 & \longrightarrow & \cC^{\tau(i)}_0 \\
& x & \longmapsto & P_{\tau(i)}\circ f (0,...,\stackbin{i}{x},...,0).
\end{array}
\]
We can extend $f_i$ to $\cC^i$ by defining
\[
\begin{array}{llll}
f'_i: & \cC^i & \longrightarrow & \cC^{\tau(i)} \\
& 1 & \longmapsto & 1 \\
& x \in \cC^i_0 & \longmapsto & f_i(x).
\end{array}
\]
Then, for the isomorphism
\[
\begin{array}{llll}
f': & \cC^1 \times \cdots \times \cC^n & \longrightarrow & \cC^1 \times \cdots \times \cC^n \\
& (0,...,\stackbin{i}{1},...,0) & \longmapsto & (0,...,\stackbin{\tau(i)}{1},...,0) \\
& (0,...,\stackbin{i}{x},...,0) & \longmapsto & (0,...,\stackbin{\tau(i)}{f_i(x)},...,0) \mbox{ for
$x\in \cC^i_0$}
\end{array}
\]
we have that $\varphi (\overline{\FF})(f') =f$. Hence $\varphi(\overline{\FF})$ is surjective and therefore $\varphi(\overline{\FF})$ is an isomorphism.

We will prove that $d \varphi$ is an isomorphism. Since $0 \times \cdots \times \cC^i\times \cdots \times 0$ is an ideal of $\cC^1 \times \cdots \times \cC^n$ for all $i=1,...,n$, it follows
$\Der (\cC^1 \times \cdots \times \cC^n) = \Der (\cC^1) \times \cdots \times  \Der(\cC^n)$.
By \cite[Proposition 3.6]{MPP01} we have
$\Der (\cC^1_0 \times \cdots \times \cC^n_0) \simeq \Der (\cC^1) \times \cdots \times \Der(\cC^n)$.
Therefore
$\Der (\cC^1_0 \times \cdots \times \cC^n_0) \simeq \Der (\cC^1 \times \cdots \times \cC^n)$.

Last part is Lemma \ref{le:grad_involution_tensor_Cayley}.
\end{proof}

\begin{remark}\label{re.forms.CxC}
The above result allows to classify forms of the structurable algebra $\cC\otimes\cC$ in a direct way, this is because it is equivalent to classify forms of $\cC\times\cC$. Forms of the tensor product of a Cayley algebra and a Hurwitz algebra were already described in \cite[Proposition 7.9]{AF92}.

\end{remark}

\subsection[Tensor product of two Hurwitz algebras]{Gradings on the tensor product of two Hurwitz Algebras}\label{subsection.Tensor.gradings.on.tensor.of.Hurwitz}

We want to find involution preserving gradings on the algebra $(\cC \otimes \cB, -)$ where $\cC$ is a Cayley algebra, $\cB$ is a Hurwitz algebra and $-$ is the tensor product of their involutions.
\textit{From now on by grading we will refer to involution preserving grading, unless indicated otherwise, and if there is no confusion we will omit the involution}.
First we give the classification, up to equivalence and isomorphism, of gradings on $\cC\otimes\cB$ where $\cC$ is a Cayley algebra and $\cB$ is a Hurwitz algebra of dimension 2 and 4. The case where $\cB$ is also a Cayley algebra will be left to Section \ref{subsection.Tensor.tensor.product}.

\vspace*{0.5cm}

First we will see some interesting graded subspaces.

\begin{lemma}\label{le:grad.subspaces}

Let $\Gamma$ be a $G$-grading on an algebra $\cA$ for a group $G$, then the following spaces are $G$-graded subspaces of $\cA$

\begin{enumerate}
\item[a)] $N_{alt}(\cA)$ (see Definition \ref{df:3.1.mpp}),

\item[b)] $\cS(\cA,-)$ and $\cH(\cA,-)$,  (see Definition \ref{df:structurable.algebra}) if $\cA= (\cA, -)$ is an algebra with involution,

\item[c)] $[\cA, \cA]$,

\item[d)] the center of $\cA$ (i.e., $Z (\cA):= \lbrace x \in \cA : xy = yx, \hspace*{0.2cm} (x,y,z)= (y,x,z) = (y,z,x) =0, \hspace*{0.2cm} \forall y,z \in \cA \rbrace $),

%\item[e)] the associator of $\cA$ (i.e., $As (\cA):= \lbrace x \in \cA : (x,y,z)= (y,x,z) = (y,z,x)=0, \hspace*{0.2cm} \forall y,z \in \cA \rbrace $),

\item[e)] $J(\cA, \cA, \cA) = \mathrm{span}\lbrace [[x,y],z]+ [[z,x],y] + [[y,z],x] : x,y,z \in \cA \rbrace$,

\item[f)] $\mathcal{D}= \lbrace x \in \cA : J(x, \cA, \cA) = \lbrace 0 \rbrace \rbrace$.

%\item[h)] $\cB':=\lbrace x \in \cA : [x,\cB]= \lbrace 0 \rbrace \rbrace$ if $\cB$ is a graded subspace of $\cA$.
\end{enumerate}
\end{lemma}

\begin{proof}
a) Let $a$ be in $N_{alt} (\cA)$ then there exist $a_i \in \cA_{g_i}$ for $i=1,...,n$ and $g_i\in G$ with $g_i\neq g_j$ if $i\neq j$ such that $a= \sum_{i=1}^n a_i$. For all homogeneous elements $x,y\in \cA$ we have
\[
\left( \sum_{i=1}^n a_i,x,y\right)= - \left(x,\sum_{i=1}^n a_i,y\right)=\left( x,y,\sum_{i=1}^n a_i\right)
\]
which is the same that
\[
\sum_{i=1}^n (a_i x) y - \sum_{i=1}^n a_i (xy)= -  \sum_{i=1}^n (x a_i) y + \sum_{i=1}^n x (a_i y)  \\
= \sum_{i=1}^n (xy)a_i- \sum_{i=1}^n x (y a_i).
\]
Then $(a_i,x,y)=-(x,a_i,y)=(x,y,a_i)$
for all $i=1,...,n$. Since this is satisfied for all homogeneous elements, it is satisfied for all $x,y\in \cA$. Then $a_i\in N_{alt} (\cA)$ for all $i=1,...,n$. The proofs of b), d) and f) are similar.

c) As $\cA=\bigoplus_{g\in G} \cA_g$ we have $[\cA,\cA]=\sum_{g,h\in G} [\cA_g,\cA_h]$
where each $[\cA_g,\cA_h]$ is a graded subspace because it is contained in $\cA_{gh}$.
The proof of e) is similar. 

\end{proof}

Next result shows that any grading on the tensor product of a Cayley algebra and a quaternion algebra preserves the involution. Recall that if $(\cA,-)$ is an algebra with involution, then $\Aut(\cA,-)$ denotes the group of involution preserving automorphisms in $\cA$.

\begin{lemma}\label{le:aut_cayley_tensor_quaternions}
Let $\cC$ be a Cayley algebra and let $\cH$ be a Hurwitz algebra of dimension $4$. Then
$N_{alt} (\cC\otimes \cH)= \cC\otimes 1 + 1\otimes \cH$ and
\[
\AAut(\cC\otimes\cH)=\AAut(\cC\otimes\cH,-)
\]
where $-$ is the tensor product of the involutions of $\cC$ and $\cH$. 
\end{lemma}

\begin{proof}
Recall $\cH$ is associative. For all $x,y,z\in \cC$ and $u,v,w\in\cH$ we have
\[
\begin{split}
(x\otimes u, y\otimes v, z\otimes w) & =(xy\otimes uv)(z\otimes w)-(x\otimes u)(yz\otimes vw)\\
& =(xy)z\otimes uvw- x(yz)\otimes uvw = (x,y,z)\otimes uvw.
\end{split}
\]
For all $y,z\in \cC$ and $u,v,w\in\cH$ we have
$(1\otimes u, y\otimes v, z\otimes w)= (1,y,z)\otimes uvw=0$, 
$(y\otimes v, 1\otimes u, z\otimes w)=0$ and $(y\otimes v, z\otimes w, 1\otimes u)=0$.
Then $1\otimes \cH \subseteq N_{alt}(\cC\otimes \cH)$.
For all $x,y,z\in \cC$ and $v,w\in\cH$ we have
$(x\otimes 1, y\otimes v, z\otimes w)= (x,y,z)\otimes vw$, $(y\otimes v, x\otimes 1, z\otimes w)=(y,x,z)\otimes vw$ and $(y\otimes v, z\otimes w, x\otimes 1)=(y,z,x)\otimes vw$.
Since $\cC$ is alternative, it follows that $(x,y,z) = -(y,x,z) = (y,z,x)$ and then $\cC\otimes 1 \subseteq N_{alt}(\cC\otimes \cH)$.
Therefore $1\otimes \cH +\cC\otimes 1 \subseteq N_{alt}(\cC\otimes \cH)$.

In order to prove the reverse containment we will consider the $(\Z/2)^5$-grading on $\cC\otimes \cH$ formed by the $(\Z/2)^3$-grading on $\cC$ and the $(\Z/2)^2$-grading on $\cH$ (both gradings induced by the Cayley-Dickson doubling process), such grading is explicitly given later in this section. Notice that each homogeneous component in such grading has dimension 1. By Lemma \ref{le:grad.subspaces} a) $N_{alt}(\cC\otimes \cH)$ is $(\Z/2)^5$-graded. The induced $(\Z/2)^5$-grading is given by
\[
\Gamma: N_{alt}(\cC\otimes \cH) =\bigoplus_{g\in (\Z/2)^5} \left( N_{alt}(\cC\otimes \cH)\cap (\cC\otimes \cH)_g \right).
\]
Therefore, each homogeneous component in such grading has dimension 1.
Consider the basis $\{ 1,i,j,k \}$ of $\cH$ where every element of the basis is homogeneous. Suppose there exist $e\neq a\in (\Z/2)^3$ and $e\neq b\in (\Z/2)^2$ such that $\cC_a\otimes\cH_b\subseteq N_{alt}(\cC\otimes \cH)$. Without loss of generality suppose $\cH_b =\FF i$, then $x\otimes i \in N_{alt}(\cC\otimes \cH)$ for $x\in \cC \setminus \FF 1$. For all $y,z\in \cC$ and $u,v\in\cH$ we have
\[
(x\otimes i, y\otimes u, z\otimes v)= -(y\otimes u, x\otimes i, z\otimes v)= (y\otimes u, z\otimes v, x\otimes i)
\]
which is the same that $(x,y,z)\otimes iuv= -(y,x,z)\otimes uiv= (y,z,x)\otimes uvi$.
If we take $u=v=j$ we have $(x,y,z)\otimes i = (y,x,z)\otimes i$ and since $\cC$ is alternative, we have $(x,y,z)=(y,x,z)=0$ for all $y,z\in \cC$ and then $x\in \FF 1$ which is a contradiction. Therefore $1\otimes \cH +\cC\otimes 1 \supseteq N_{alt}(\cC\otimes \cH)$.

Now we will prove that for any object $R$ in $\Alg_{\FF}$ we have
\[
\Aut_{R-\mathrm{alg}}(\cC\otimes\cH\otimes R)=\Aut_{R-\mathrm{alg}}(\cC\otimes\cH\otimes R,-).
\]
$\supseteq$ is clear. Take $\varphi\in \Aut_{R-\mathrm{alg}}(\cC\otimes\cH\otimes R)$. Observe that $\varphi$ preserves $[N_{alt}(\cC\otimes\cH\otimes R),N_{alt}(\cC\otimes\cH\otimes R)]$ and
\[
\begin{split}
[N_{alt}(\cC\otimes\cH\otimes R),N_{alt}(\cC\otimes\cH\otimes R)] &
= [(\cC\otimes1 + 1\otimes\cH)\otimes R), (\cC\otimes 1 + 1\otimes\cH)\otimes R)]\\
& = \cC_0 \otimes 1\otimes R + 1\otimes\cH_0\otimes R.
\end{split}
\]
It is straightforward to prove that $-\otimes id_R = -id_{\cC_0 \otimes 1\otimes R + 1\otimes\cH_0\otimes R}$. Then $\varphi\circ(-\otimes id_R)= (-\otimes id_R)\circ \varphi$ in $\cC_0 \otimes 1\otimes R + 1\otimes\cH_0\otimes R$ and, since $\cC_0 \otimes 1\otimes R + 1\otimes\cH_0\otimes R$ generates $\cC\otimes\cH\otimes R$ (see Remark \ref{re:C_0}), also in $\cC\otimes\cH\otimes R$. Then
\[
\AAut(\cC\otimes\cH)=\AAut(\cC\otimes\cH,-).
\]
\end{proof}

\begin{remark}\label{re:involution_non_important}
Let $\cC, \cC^1$ and $\cC^2$ be Cayley algebras and let $\cH$ be a quaternion algebra.
All gradings on $\cC^1\otimes\cC^2$  and $\cC\otimes\cH$ preserve the involution.
This follows from Theorem \ref{th:iso.aut.schemes.cayley} and Lemma \ref{le:aut_cayley_tensor_quaternions} since any automorphism in these algebras preserves the involution. Notice that this is not the case for the tensor product of a Cayley algebra $\cC$ and a Hurwitz algebra $\cK$ of dimension $2$, this is easy to see from the fact that
$\cC\otimes\cK \simeq \cC\times\cC$  as algebras and, by Theorem \ref{th:iso.aut.schemes.cayley}, gradings on $\cC\times\cC$ are in correspondence with gradings on $\cC\otimes\cC$. So, if gradings on $\cC\otimes\cK$ did not depend on the involution we would have a correspondence between gradings on $\cC\otimes\cK$ and $\cC\otimes\cC$. Theorem \ref{th:grad.tensor.prod.dim=2,4} will show that this is not possible.
\end{remark}

\begin{remark}\label{re:cayley.simple}
It is well known that the Hurwitz algebras of dimension 4 and 8 are simple.
Moreover, since $\ch \FF \neq 2$, for the Hurwitz algebra $\cB$ where $\dm(\cB)= 4,8$ we have that $\cB_0$ is a simple subalgebra of $\cB$ under the  product given by the commutator.
\end{remark}

Next result will be useful in the proof of Theorem \ref{th:grad.tensor.prod.dim=2,4}. 

\begin{lemma}\label{le:graded.ideals}
Let $\cA=\cA_1\oplus \cA_2$ be a finite dimensional $G$-graded algebra and let $\cA_1$ and $\cA_2$ be central simple ideals of $\cA$ such that $\dim \cA_1\neq \dim \cA_2$. Then $\cA_1$ and $\cA_2$ are graded ideals of $\cA$.
\end{lemma}

\begin{proof}
Suppose $\cA$ is graded-simple (i.e., $\cA^2\neq 0$ and the only graded ideals of $\cA$ are $0$ and $\cA$). 
We extend $\cA$ to the algebraic closure $\cA\otimes_{\FF} \overline{\FF}:=\overline{\cA}$ which is graded-simple as well. Set $\overline{\cA_i}= \cA_i\otimes_{\FF} \overline{\FF}$ for $i=1,2$.
From \cite[Chapter 1, Section 2]{EM94} we get that $C(\overline{\cA})\simeq C(\overline{\cA_1})\oplus C(\overline{\cA_2}) \simeq \overline{\FF} \times \overline{\FF}$, where $C(\overline{\cA})$ is the centroid of $\overline{\cA}$. By \cite[Lemma 4.2.3 ii)]{ABFP08}  we have that $C(\overline{\cA})$ is a $G$-graded algebra. 
Since $\cA$ is finite dimensional we can apply \cite[Lemma 4.3.4]{ABFP08} and we get that $\overline{\cA}$ is graded-central-simple (graded-simple and \textit{graded-central}: $C(\overline{\cA})_e=\FF 1$ where $e$ is the neutral element of $G$ and $1$ is the unity of $C(\overline{\cA})$).
It follows that $\overline{\FF}(Id_{\overline{\cA_1}},-Id_{\overline{\cA_2}})= C(\overline{\cA})_h$ for some $h\in G$ such that $h\neq e$ and $h^2=e$.
Let $\Gamma_{C(\overline{\cA})}$ be the grading induced on $C(\overline{\cA})$. Take $H:=\Supp (\Gamma_{C(\overline{\cA})})= \lbrace e,h \rbrace$, $\overline{G}:= G/H$ and let $\pi: G\rightarrow \overline{G}$ be the canonical projection. 
By \cite[Theorem 4.1]{CE18}, we have $ L_{\pi}(\overline{\cA_1})\simeq \overline{\cA} \simeq L_{\pi}(\overline{\cA_2})$ as $G$-graded algebras. By \cite[Theorem 3.5 (5)]{CE18} we have $\overline{\cA_1}\simeq \overline{\cA_2}$ which is a contradiction. Therefore $\overline{\cA}$ has nontrivial graded ideals which are $\overline{\cA_1}$ and $\overline{\cA_2}$ because they are simple factors. Finally we have that $\cA_1$ and $\cA_2$ are graded ideals of $\cA$.

\end{proof}

Next theorem shows what the gradings on $\cC \otimes \cB$ look like, where $\cC$ is a Cayley algebra and $\cB$ is a Hurwitz algebra of dimension 2 or 4. It also classifies gradings on these algebras up to isomorphism.

\begin{theorem}\label{th:grad.tensor.prod.dim=2,4}
Let $\cA=(\cC \otimes \cB, -)$ be the algebra with involution where $\cC$ is a Cayley algebra and $\cB$ is a Hurwitz algebra of dimension 2 or 4 and $-$ denotes the tensor product of the involutions of $\cC$ and $\cB$. Then $\Gamma: \cA= \oplus_{g \in G} \cA_g$ is a $G$-grading on $\cA$ if and only if there exist $G$-gradings
\[
\Gamma^1: \cC = \bigoplus\limits_{g \in G} \cC_g \mbox{\hspace*{0.1cm} and \hspace*{0.1cm}}
\Gamma^2: \cB = \bigoplus\limits_{g \in G} \cB_g
\]
such that for all $g \in G$
\[
\cA_g = \bigoplus\limits_{g_1, g_2 \in G : g_1 g_2 = g} \cC_{g_1} \otimes \cB_{g_2}.
\]
Moreover, two $G$-gradings $\Gamma$ and $\Gamma'$ on $\cA$ are isomorphic if and only if so are $\Gamma^1$ and $(\Gamma')^1$ on $\cC$ and $\Gamma^2$ and $(\Gamma')^2$ on $\cB$.
\end{theorem}

\begin{proof}
We have that $\cS(\cC \otimes \cB, -)= \cC_0 \otimes 1 \oplus 1 \otimes \cB_0$ (see definitions \ref{df:antisymmetric.elements} and \ref{df:structurable.algebra}).

\noindent $\Rightarrow)$
By Remark \ref{re:cayley.simple} and Lemma \ref{le:grad.subspaces} b) we get that $\cS:=\cS(\cC \otimes \cB,-)$ is a $G$-graded subalgebra of $\cA$ with the product given by the commutator. Set
\[
\Gamma_{\cS}:= \Gamma |_{\cS}: \cS = \bigoplus_{g\in G} (\cA_g \cap \cS).
\]
Suppose $\dim (\cB)=2$.
We will prove that $[\cS, \cS]= \cC_0 \otimes 1$ and $Z(\cS) = 1 \otimes \cB_0$.
We have
\[
\begin{split}
[\cS, \cS]= & [\cC_0 \otimes 1 \oplus 1 \otimes \cB_0, \cC_0 \otimes 1 \oplus 1 \otimes
\cB_0 ]\\
= & [\cC_0 \otimes 1 , \cC_0 \otimes 1] + [\cC_0 \otimes 1 , 1 \otimes \cB_0] +
[1 \otimes \cB_0 , \cC_0 \otimes 1 ] + [1 \otimes \cB_0 , 1 \otimes \cB_0] \\
= & [\cC_0 , \cC_0] \otimes 1 + 1 \otimes [\cB_0 , \cB_0] \\
= & \cC_0 \otimes 1
\end{split}
\]
where the last equality follows from the fact that $\dm(\cB_0)=1$.
Take $a= c \otimes 1 + 1 \otimes b \in \cS$ with $c \in \cC_0$ and $b \in \cB_0$, then
\[
\begin{split}
a\in Z(\cS) & \Leftrightarrow [a,s] = 0 \mbox{ for all $s \in \cS$}\\
& \Leftrightarrow [c \otimes 1 + 1 \otimes b, s_1 \otimes 1 + 1 \otimes s_2]= 0
 \mbox{ for all $s_1 \in \cC_0$ and $s_2 \in \cB_0$} \\
& \Leftrightarrow [c, s_1] \otimes 1 = 0 \mbox{ for all $s_1 \in \cC_0$. }
\end{split}
\]
Since $\cC_0$ is simple under the commutator, $Z(\cC_0)= \lbrace 0 \rbrace$. Therefore $Z(\cS)= 1 \otimes \cB_0$.
By Lemma \ref{le:grad.subspaces} c) and d), $[\cS, \cS]= \cC_0 \otimes 1$ and
$Z(S) = 1 \otimes \cB_0$ are
$G$-graded subspaces of $\cS$ with the gradings induced by $\Gamma_{\cS}$:
\[
\Gamma_{\cC_0 \otimes 1}:= \Gamma_{\cS} |_{\cC_0 \otimes 1} \mbox{ and }
\Gamma_{1 \otimes \cB_0}:= \Gamma_{\cS} |_{1 \otimes \cB_0}.
\]
Consider the following isomorphisms
\[
1) \hspace*{0.5cm}
\begin{array}{llll}
\varphi_1: & \cC_0 \otimes 1 & \rightarrow  & \cC_0 \\
   & x \otimes 1 & \mapsto &  x
\end{array}
\mbox{ and }
\begin{array}{llll}
\varphi_2: & 1 \otimes \cB_0 & \rightarrow  & \cB_0 \\
   & 1 \otimes y & \mapsto &  y.
\end{array}
\]
Then we have gradings on $\cC_0$ and $\cB_0$ given by
\[
2) \hspace*{0.5cm}
\Gamma_{\cC_0}: \cC_0 = \bigoplus_{g \in G} (\cC_0)_g, \mbox{ where } (\cC_0)_g= \varphi_1
((\cC_0 \otimes 1)_g)
\]
and
\[
3) \hspace*{0.5cm}
\Gamma_{\cB_0}: \cB_0 = \bigoplus_{g \in G} (\cB_0)_g, \mbox{ where } (\cB_0)_g= \varphi_2
((1 \otimes \cB_0)_g).
\]
By \cite[Corollary 4.25]{EKmon}, the decomposition
\[
\Gamma_{\cC}: \cC = \bigoplus_{g \in G} \cC_g
\]
where $\cC_e := \FF1 \oplus (\cC_0)_e$ and $\cC_g := (\cC_0)_g$ for $g \neq e$, where $e$ is the neutral element of $G$, is a $G$-grading on $\cC$. We have the $G$-grading on $\cB$
\[
\Gamma_{\cB}: \cB = \bigoplus_{g \in G} \cB_g
\]
given by $\cB_e = \FF 1$ and $\cB_g= (\cB_0)_g$, notice that $g^2 =e$ for $g\in \Supp \Gamma_{\cB_0}$.

Now suppose $\dim (\cB)=4$. By Lemma \ref{le:graded.ideals}, $\cC_0 \otimes 1$ and $1 \otimes \cB_0$ are graded. Consider the gradings (induced from the grading on $\cS$):
\[
\Gamma_{\cC_0 \otimes 1}:= \Gamma_{\cS} |_{\cC_0 \otimes 1} \hspace*{0.2cm} \mbox{ and }\hspace*{0.2cm}
\Gamma_{1 \otimes \cB_0}:= \Gamma_{\cS} |_{1 \otimes \cB_0}.
\]
Consider again the isomorphisms $\varphi_1$ and $\varphi_2$ from 1) and we have
the $G$-gradings on $\cC_0 \otimes 1$ and $1 \otimes \cB_0$ given by 2) and 3), respectively.
By \cite[Corollary 4.25]{EKmon} we have that the decompositions
\[
\Gamma_{\cC}: \cC = \bigoplus_{g \in G} \cC_g \hspace*{0.2cm}\mbox{ and } \hspace*{0.2cm} \Gamma_{\cB}: \cB = \bigoplus_{g \in G} \cB_g,
\]
where $\cC_e := \FF1 \oplus (\cC_0)_e$, $\cB_e := \FF1 \oplus (\cB_0)_e$, $\cC_g := (\cC_0)_g$ and $\cB_g := (\cB_0)_g$ for $g \neq e$, are $G$-gradings. Now we will prove that, effectively
\[
(\cC \otimes \cB)_g = \bigoplus\limits_{g_1, g_2 \in G : g_1 g_2 = g} \cC_{g_1} \otimes
\cB_{g_2}.
\]
For $h,k \in G$:
\begin{itemize}
\item $\cC_{h} \otimes \cB_{k} = (\cC_0)_{h} \otimes (\cB_0)_{k} = (\cC_0 \otimes
1)_{h} (1 \otimes \cB_0)_{k} \subseteq (\cC \otimes \cB)_{h} (\cC \otimes \cB)_{k}
\subseteq (\cC \otimes \cB)_{hk}$, if $h \neq e \neq k$;

\item $\cC_{h} \otimes \cB_{k} =
(\FF 1 \oplus (\cC_0)_e) \otimes (\cB_0)_{k}  = ((\FF 1 \oplus (\cC_0)_e) \otimes 1) (1 \otimes \cB_0)_{k} = ((\FF 1 \otimes 1) \oplus ((\cC_0)_e \otimes 1)) (1 \otimes \cB_0)_{k}
 \subseteq (\cC \otimes \cB)_{e} (\cC \otimes \cB)_{k} \subseteq (\cC \otimes
\cB)_{k}$, if $h = e$ and $k \neq e$;

\item $\cC_{h} \otimes \cB_{k} =
(\cC_0)_h \otimes (\FF 1 \oplus (\cB_0)_e) = (\cC_0 \otimes 1)_{h} (1 \otimes (\FF 1 \oplus
(\cB_0)_e)) = (\cC_0 \otimes 1)_{h} ((1 \otimes \FF 1) \oplus (1 \otimes (\cB_0)_e))
\subseteq (\cC \otimes \cB)_{h} (\cC \otimes \cB)_{e} \subseteq (\cC \otimes
\cB)_{h}$ if $h \neq e$ and $k = e$.

\end{itemize}
Therefore $\cC_{h} \otimes \cB_{k} \subseteq (\cC \otimes \cB)_{hk}$. Since
\[
\bigoplus_{h,k \in G} \cC_{h} \otimes \cB_{k} = \cC \otimes \cB = \bigoplus_{g \in G}
(\cC \otimes \cB)_{g}
\]
and
\[
\bigoplus_{h,k \in G : hk=g}\cC_{h} \otimes \cB_{k} \subseteq (\cC \otimes \cB)_{g},
\]
we get that
\[
\cA_g = \bigoplus\limits_{g_1, g_2 \in G : g_1 g_2 = g} \cC_{g_1} \otimes \cB_{g_2}.
\]
$\Leftarrow$)
This defines a grading on a tensor product of algebras (see \cite[Chapter 1, Section 1]{EKmon}).

For the last part consider a graded isomorphism $\psi: \cC\otimes \cB\rightarrow \cC\otimes \cB$. Then the restriction to $\cS$, $\psi: \cC_0 \otimes 1 \oplus 1\otimes \cB_0\rightarrow \cC_0 \otimes 1 \oplus 1\otimes \cB_0$ is a graded isomorphism as well. Notice that $\cC_0 \otimes 1$ and $1\otimes \cB_0$ are the only simple ideals of $\cS$. Since every isomorphism sends simple ideals to simple ideals and $\cC_0 \otimes 1$ and $1\otimes \cB_0$ are graded then we can restrict $\psi$ again and get graded isomorphisms $\cC_0\rightarrow \cC_0$ and $\cB_0\rightarrow \cB_0$.

\end{proof}

Assume the base field is algebraically closed. The following gradings are the only fine gradings, up to equivalence, on the tensor product of a Cayley algebra $\cC$ and a Hurwitz algebra $\cB$ of dimension 2 and 4. This follows from Theorem \ref{th:grad.tensor.prod.dim=2,4} and the fact that they are gradings by their universal groups.

Suppose $\dm(\cB)=2$ and take $\cK:=\cB$.
\begin{enumerate}
\item The $(\Z/2)^4$-grading formed by the $(\Z/2)^3$-grading induced by the Cayley-Dickson doubling process on $\cC$ (with basis $\{ 1,u,v,w,uv,uw,vw,(uv)w \}$ and homogeneous components given by Equation \eqref{eq:doubling.process.grading}) and the $\Z/2$-grading induced by the Cayley-Dickson doubling process on $\cK=CD(\FF,\alpha)$ $=\FF 1\oplus\FF u$ for $\alpha\in\FF$ (with basis $\{ 1,u\}$ and homogeneous components given by Equation \eqref{eq:C-D.grading.dim2}), with homogeneous components given by
\[
(\cC\otimes\cK)_{(g,h)}=\cC_g \otimes \cK_h, \hspace*{0.2cm} g\in (\Z/2)^3, \hspace*{0.2cm} h \in \Z/2.
\]
\item The $\Z^2\times\Z/2$-grading formed by the Cartan $\Z^2$-grading on $\cC$ (with basis  $\{ e_1,e_2,u_1,u_2,u_3,v_1,v_2,v_3 \}$ and homogeneous components given by Equation \eqref{eq:cartan.grading}) and the $\Z/2$-grading induced by the Cayley-Dickson doubling process on $\cK=CD(\FF,\alpha)=\FF 1\oplus\FF u$ for $\alpha\in\FF$ (with basis $\{ 1,u\}$ and homogeneous components given by Equation \eqref{eq:C-D.grading.dim2}), with homogeneous components given by
\[
(\cC\otimes\cK)_{(g,h)}=\cC_g \otimes \cK_h, \hspace*{0.2cm} g\in \Z^2, \hspace*{0.2cm} h \in \Z/2.
\]
\end{enumerate}

Suppose $\dm(\cB)=4$ and take $\cH:=\cB$.
\begin{enumerate}
\item The $(\Z/2)^5$-grading formed  by the $(\Z/2)^3$-grading induced by the Cayley-Dickson doubling process on $\cC$ (with basis $\{ 1,u,v,w,uv,uw,vw,(uv)w \}$ and homogeneous components given by Equation \eqref{eq:doubling.process.grading}) and the $(\Z/2)^2$-grading induced by the Cayley-Dickson doubling process on $\cH$ (with basis $\{ 1,u,v,uv\}$ and homogeneous components given by \eqref{eq:C-D.grading.quaternions}), with homogeneous components given by
\[
(\cC\otimes\cH)_{(g,h)}=\cC_g \otimes \cH_h, \hspace*{0.2cm} g\in (\Z/2)^3, \hspace*{0.2cm} h \in (\Z/2)^2.
\]
\item The $(\Z/2)^3\times\Z$-grading formed by the $(\Z/2)^3$-grading induced by the Cayley-Dickson doubling process on $\cC$ (with basis $\{ 1,u,v,w,uv,uw,vw,$ $(uv)w \}$ and homogeneous components given by Equation \eqref{eq:doubling.process.grading}) and the Cartan $\Z$-grading on $\cH$ (with basis $\{ e_1,e_2,u_1,v_1\}$ and homogeneous components given by Equation \eqref{eq:cartan.grading.quaternions}), with homogeneous components given by
\[
(\cC\otimes\cH)_{(g,h)}=\cC_g \otimes \cH_h, \hspace*{0.2cm} g\in (\Z/2)^3, \hspace*{0.2cm} h \in \Z.
\]
\item The $\Z^2\times(\Z/2)^2$-grading formed by the Cartan $\Z^2$-grading on $\cC$ (with basis $\lbrace e_1,e_2,u_1,u_2,u_3,v_1,v_2,v_3 \rbrace$ and homogeneous components given by Equation \eqref{eq:cartan.grading}) and the $(\Z/2)^2$-grading induced by the Cayley-Dickson doubling process on $\cH$ (with basis $\lbrace 1,u,v,uv\rbrace$ and homogeneous components given by Equation \eqref{eq:C-D.grading.quaternions}), with homogeneous components given by
\[
(\cC\otimes\cH)_{(g,h)}=\cC_g \otimes \cH_h, \hspace*{0.2cm} g\in \Z^2, \hspace*{0.2cm} h \in (\Z/2)^2.
\]
\item The $\Z^3$-grading formed by the Cartan $\Z^2$-grading on $\cC$ (with basis $\{ e_1,e_2,$

\noindent $u_1,u_2,u_3,v_1,v_2,v_3 \}$ and homogeneous components given by Equation \eqref{eq:cartan.grading}) and the Cartan $\Z$-grading on $\cH$ (with basis $\{ e_1,e_2,u_1,v_1\}$ and homogeneous components given by Equation \eqref{eq:cartan.grading.quaternions}), with homogeneous components given by
\[
(\cC\otimes\cH)_{(g,h)}=\cC_g \otimes \cH_h, \hspace*{0.2cm} g\in \Z^2, \hspace*{0.2cm} h \in \Z.
\]
\end{enumerate}

\subsection{The direct product of two Cayley algebras}\label{subsection.Tensor.direct.product}

In order to find the gradings on the tensor product of two Cayley algebras we will start by finding gradings on their direct product. Then we will use an isomorphism of schemes to obtain gradings on the tensor product.
In this section we will use some results from \cite{CE18} and \textit{assume that the base field is algebraically closed}, recall that in this case any finite-dimensional simple algebra is automatically central-simple (this is a consequence of \cite[Theorem 10.1]{Jac78}).

\vspace{0.2cm}
Consider the following notation:
\begin{itemize}
\item Let $\gamma=(g_1, g_2, g_3)$ be a triple of elements in a group $G$ with $g_1 g_2 g_3 =e$. Denote by
$\Gamma^1_{\cC} (G,\gamma)$ the $G$-grading on $\cC$ induced from $\Gamma^1_{\cC}$ by the
homomorphism $\Z^2 \rightarrow G$ sending $(1,0)$ to $g_1$ and $(0,1)$ to $g_2$. For two such
triples, $\gamma$ and $\gamma'$ we will write $\gamma \sim \gamma'$ if there exists $\pi \in Sym(3)$
such that $g'_i = g_{\pi(i)}$ for all $i=1,2,3$ or $g'_i = g^{-1}_{\pi(i)}$ for all $i=1,2,3$.

\item Let $H \subset G$ be a subgroup isomorphic to $(\Z/2)^3$. Then $\Gamma^2_{\cC}$ may be regarded as a $G$-grading with support $H$. We denote this $G$-grading by $\Gamma^2_{\cC} (G, H)$.
\end{itemize}

\cite[Theorems 3.7 and 4.1]{CE18} show that given any abelian group $G$, any $G$-grading on $\cC\times\cC$ making it a graded-simple algebra (i.e., the two copies of $\cC$ are not graded ideals) is isomorphic to the grading on a loop algebra $L_\pi(\cC)$, where $\pi:G\rightarrow \overline{G}$ is a surjective group homomorphism with $\ker\pi$ of order $2$ ($\ker\pi=\langle h\rangle$ with $h$ of order $2$) obtained from a grading $\overline{\Gamma}$ on $\cC$. We will denote $\overline{g}:=\pi(g)$ for $g\in G$. The loop algebra is isomorphic to $\cC\times\cC$ by means of the isomorphism in \cite[Theorem 3.7]{CE18} which allows us to transfer easily the grading on $L_\pi(\cC)$ to $\cC\times\cC$.

If $\overline{\Gamma}$ is isomorphic to $\Gamma^1_{\cC} (\overline{G},\overline{\gamma})$, for a triple of elements $\overline{\gamma}=(\bar{g}_1,\bar{g}_2,\bar{g}_3)$ in $\overline{G}$, the corresponding grading on $\cC\times\cC$ will be denoted by $\Gamma^1_{\cC\times\cC}(G,h,\overline{\gamma})$. While if $\overline{\Gamma}$ is isomorphic to $\Gamma^2_{\cC} (\overline{G}, \overline{H})$ for $\overline{H}:=\pi(H)$, where $H$ is a subgroup of $G$ such that $\overline{H}$ is isomorphic to $(\Z/2)^3$, the corresponding grading on $\cC\times\cC$ will be denoted by $\Gamma_{\cC\times\cC}^2(G,h,\overline{H})$.

The gradings $\Gamma^1_{\cC\times\cC}(G,h,\overline{\gamma})$ and $\Gamma_{\cC\times\cC}^2(G,h,\overline{H})$ are quite simple to describe if the surjective group homomorphism $\pi:G\rightarrow \overline{G}$ splits. That is, there is a section $s:\overline{G}\rightarrow G$ of $\pi$ which is a group homomorphism. In this case, $G=s(\overline{G})\times \langle h \rangle$ and the nontrivial character on $\ker \pi= \langle h \rangle$ ($\chi(h)=-1$) extends to a character $\chi$ on $G$ with $\chi(g)=1$ for any $g\in s(\overline{G})$. The isomorphism in \cite[Theorem 3.7]{CE18} becomes the isomorphism
\[
\begin{split}
\Phi:L_\pi(\cC)&\longrightarrow \cC\times\cC\\
    x\otimes g&\mapsto \ \bigl(x,\chi(g)x\bigr)
\end{split}
\]
for $g\in G$ and  $x\in \cC_{\pi(g)}$. Thus, with $g_i=s(\overline{g_i})$ for $i=1,2,3$, the $G$-grading $\Gamma^1_{\cC\times\cC}(G,h,\overline{\gamma})$ is determined by:
\begin{equation}\label{eq:gamma1_h}
\begin{array}{ll}
(\cC\times\cC)_{e}=\FF (e_1,e_1) \oplus \FF (e_2,e_2), &(\cC\times\cC)_{h}=\FF (e_1,-e_1) \oplus \FF (e_2,-e_2),\\
(\cC\times\cC)_{g_1}=\FF (u_1,u_1), & (\cC\times\cC)_{g_1 h}=\FF (u_1,-u_1),\\
(\cC\times\cC)_{g_2}=\FF (u_2,u_2), & (\cC\times\cC)_{g_2 h}=\FF (u_2,-u_2),\\
(\cC\times\cC)_{g_1 g_2}=\FF (v_3,v_3), & (\cC\times\cC)_{g_1 g_2 h}=\FF(v_3,-v_3),\\
(\cC\times\cC)_{g_1^{-1}}=\FF (v_1,v_1), & (\cC\times\cC)_{g_1^{-1}h}=\FF (v_1,-v_1),\\
(\cC\times\cC)_{g_2^{-1}}=\FF (v_2,v_2), & (\cC\times\cC)_{g_2^{-1}h}=\FF (v_2,-v_2),\\
(\cC\times\cC)_{(g_1 g_2)^{-1}}=\FF (u_3,u_3), & (\cC\times\cC)_{(g_1 g_2)^{-1}h}=\FF (u_3,-u_3).
\end{array}
\end{equation}
And the $G$-grading $\Gamma_{\cC\times\cC}^2(G,h,\overline{H})$ is determined by:
\begin{equation}\label{eq:gamma2_h}
\begin{array}{ll}
(\cC\times\cC)_{e}=\FF (1,1), & (\cC\times\cC)_{h}=\FF(1,-1),\\
(\cC\times\cC)_{g_1}=\FF (u,u), & (\cC\times\cC)_{g_1 h}=\FF (u,-u),\\
(\cC\times\cC)_{g_2}=\FF (v,v), & (\cC\times\cC)_{g_2 h}=\FF (v,-v),\\
(\cC\times\cC)_{g_3}=\FF (w,w), & (\cC\times\cC)_{g_3 h}=\FF (w,-w),\\
(\cC\times\cC)_{g_1 g_2}=\FF (uv,uv), & (\cC\times\cC)_{g_1 g_2 h}=\FF (uv,-uv),\\
(\cC\times\cC)_{g_1 g_3}=\FF (uw,uw), & (\cC\times\cC)_{g_1 g_3 h}=\FF (uw,-uw),\\
(\cC\times\cC)_{g_2 g_3}=\FF (vw,vw), & (\cC\times\cC)_{g_2 g_3 h}=\FF (vw,-vw),\\
(\cC\times\cC)_{g_1 g_2 g_3}=\FF ((uv)w,(uv)w), & (\cC\times\cC)_{g_1 g_2 g_3 h}=\FF ((uv)w,-(uv)w)
\end{array}
\end{equation}
where $\overline{H}=\langle \bar{g}_1,\bar{g}_2,\bar{g}_3 \rangle$.

The following result gives the classification of gradings, up to isomorphism, on $\cC\times \cC$ where it is graded-simple.

\begin{theorem}\label{th:iso.CxC.graded-simple}
Any $G$-grading $\Gamma$ by a group $G$ on the cartesian product $\cC\times\cC$ of two
Cayley algebras, such that $\cC\times\cC$ is graded-simple, is isomorphic to either $\Gamma^1_{\cC\times\cC}(G,h,\overline{\gamma})$ or $\Gamma_{\cC\times\cC}^2(G,h,\overline{H})$ (for an element $h$ in $G$ of order 2, a triple $\overline{\gamma}=(\bar{g}_1, \bar{g}_2, \bar{g}_3)$ of elements of $\overline{G}=G/\langle h \rangle$ and a subgroup $\overline{H} \subset \overline{G}$ isomorphic to $(\Z/2)^3$). Moreover, no grading of the first type is isomorphic to one of the second type and
\begin{itemize}
\item $\Gamma^1_{\cC\times\cC}(G,h,\overline{\gamma})$ is isomorphic to $\Gamma^1_{\cC\times\cC}(G,h',\overline{\gamma}')$ if and only if $h=h'$ and $\overline{\gamma}\sim\overline{\gamma}'$.

\item $\Gamma_{\cC\times\cC}^2(G,h,\overline{H})$ is isomorphic to $\Gamma_{\cC\times\cC}^2(G,h',\overline{H}')$ if and only if $h=h'$ and $\overline{H}=\overline{H}'$.
\end{itemize}
\end{theorem}

\begin{proof}
By \cite[Theorem 3.5 (5)]{CE18} we have that $\Gamma^1_{\cC\times\cC}(G,h,\overline{\gamma})$ is isomorphic to $\Gamma^1_{\cC\times\cC}(G,h',\overline{\gamma}')$ if and only if $h=h'$ and the associated $\overline{G}$-gradings on $\cC$, that is $\Gamma^1_{\cC}(\overline{G},\overline{\gamma})$ and $\Gamma^1_{\cC}(\overline{G},\overline{\gamma}')$ are isomorphic, which occurs if and only if $\overline{\gamma}\sim\overline{\gamma}'$ (\cite[Theorem 4.21]{EKmon}). The  proof for the grading of second type is analogous.
\end{proof}

Next result gives the classification of gradings, up to isomorphism, on $\cC\times \cC$ where it is not graded-simple.

\begin{proposition}\label{pr:product.grading.isomorphism.CxC}
Let $G$ be a group and let $\Gamma$ be a $G$-grading on the product of two Cayley algebras $\cC\times\cC$ such that it is not graded-simple, i.e., $\cC\times 0$ and $0\times\cC$ are graded ideals. Then $\Gamma$ is isomorphic to a product $G$-grading $\Gamma^1\times_G\Gamma^2$ for some $G$-gradings $\Gamma^1$ and $\Gamma^2$ on $\cC$.

Let $\Gamma^1$, $\Gamma^2$, $\Gamma'^1$ and $\Gamma'^2$ be $G$-gradings on $\cC$.
Then, the product $G$-gradings $\Gamma^1\times_G\Gamma^2$ and $\Gamma'^1\times_G\Gamma'^2$ are isomorphic if and only if $\Gamma^1 \simeq \Gamma'^1$ and $\Gamma^2 \simeq \Gamma'^2$ or $\Gamma^1 \simeq \Gamma'^2$ and $\Gamma^2 \simeq \Gamma'^1$.
\end{proposition}
\begin{proof}
First assertion follows from \cite[Theorem 4.1 (1)]{CE18}. Second assertion is trivial.
\end{proof}

Notice that the fine grading $\Gamma_{\cC}^1(\Z^2,((1,0),(0,1),(-1,-1)))$ is precisely  $\Gamma_{\cC}^1$ and the fine grading $\Gamma_{\cC}^2$ is $\Gamma^2_{\cC}(G,H)$ with $G=H=\bigl(\Z/2\bigr)^3$.

Finally we obtain the fine gradings on $\cC\times \cC$ up to equivalence.

\begin{proposition}\label{pr:CxC.equivalence}
Up to equivalence, the fine gradings on $\cC\times\cC$ are:
\begin{enumerate}
\item The product grading $\Gamma^1_{\cC} \times \Gamma^1_{\cC}$ by its universal group $\Z^2 \times \Z^2 \simeq \Z^4$.

\item The product grading $\Gamma^1_{\cC} \times \Gamma^2_{\cC}$ by its universal group $\Z^2 \times (\Z/2)^3$.

\item The product grading $\Gamma^2_{\cC} \times \Gamma^2_{\cC}$ by its universal group $(\Z/2)^3 \times (\Z/2)^3 \simeq (\Z/2)^6$.

\item The grading $\Gamma_{\cC\times\cC}^1(\Z/2\times\Z^2,(\bar{1},0,0),((1,0),(0,1),(-1,-1)))$ with universal group $U=\Z/2\times\Z^2$. The group $\overline{U}=U/\langle (\bar{1},0,0)\rangle$ is identified naturally with $\Z^2$. This grading is determined explicitly using Equation \eqref{eq:gamma1_h}.

\item The grading $\Gamma_{\cC\times\cC}^2\Bigl(\bigl(\Z/2\bigr)^4,(\bar 1,\bar 0,\bar 0,\bar 0),\bigl(\Z/2\bigr)^3\Bigr)$ with universal group $U=\bigl(\Z/2\bigr)^4$. Here the group $\overline{U}=U/\langle (\bar 1,\bar 0,\bar 0,\bar 0)\rangle$ is identified with $\bigl(\Z/2\bigr)^3$. This grading is determined explicitly using Equation \eqref{eq:gamma2_h}.

\item The grading $\Gamma_{\cC\times\cC}^2(\Z/4\times(\Z/2)^2,(\widehat{2},\bar 0,\bar 0),(\Z/2)^3)$. Here we denote by $\widehat{m}$ the class of the integer $m$ modulo $4$ and restrict the usual notation $\bar{m}$ for the class of $m$ modulo $2$. The surjective group homomorphism $\pi$ is the canonical homomorphism $\Z/4\times(\Z/2)^2\rightarrow (\Z/2)^3$, $(\widehat{m},\bar n,\bar r)\mapsto (\bar{m},\bar n,\bar r)$.

Let us give a precise description of this grading. The nontrivial character $\chi$ on $\langle h=(\widehat{2},\bar 0,\bar 0)\rangle$ extends to the character $\chi$ on $U=\Z/4\times(\Z/2)^2$ by $\chi(\widehat{m},\bar n,\bar r)=\mathbf{i}^m$, where $\mathbf{i}$ denotes a square root of $-1$ in $\FF$.

The grading on the loop algebra $L_\pi(\cC)$ is given by
\[
L_{\pi}(\cC)_{(\widehat{m},\bar n,\bar r)}= (\cC)_{(\bar{m},\bar n,\bar r)} \otimes (\widehat{m},\bar n,\bar r)
\]
for the homogeneous components $\cC_{(\bar{m} ,\bar n,\bar r)}$ in Equation \eqref{eq:doubling.process.grading}, and through the isomorphism in \cite[Theorem 3.7]{CE18} our grading
\[
\Gamma_{\cC\times\cC}^2(\Z/4\times(\Z/2)^2,(\widehat{2},\bar 0,\bar 0),(\Z/2)^3)
\]
on $\cC\times\cC$ is given by
\[
(\cC\times\cC)_{(\widehat{m},\bar n, \bar r)}=\{(x,\mathbf{i}^mx)\mid x\in \cC_{(\bar{m},\bar n,\bar r)}\}.
\]
\end{enumerate}
\end{proposition}
\begin{proof}
It follows from \cite[Corollary 5.8]{CE18}.
\end{proof}

\subsection{The tensor product of two Cayley algebras}\label{subsection.Tensor.tensor.product}

\textit{Assume in this section that the base field is algebraically closed}. We will generate gradings on the tensor product of two Cayley algebras $\cC\otimes\cC$ from the gradings we already know on $\cC\times\cC$. This is enough to classify gradings on $\cC\otimes\cC$ since there is a correspondence between gradings on $\cC\times\cC$ and gradings on $\cC\otimes\cC$ (Theorem \ref{th:iso.aut.schemes.cayley}).

We start by generating gradings on $\cC\otimes\cC$ from the gradings on $\cC\times\cC$ such that this product is graded-simple.
\vspace*{0.5cm}

Let $G$ be a group, let $h$ be an element in $G$ of order 2 and let $\overline{\gamma}=(\bar{g}_1, \bar{g}_2, \bar{g}_3)$ be a triple of elements in $\overline{G}:=G/\langle h \rangle$. Consider the $G$-grading $\Gamma^1_{\cC\times\cC}(G,h,\overline{\gamma})$ (on $\cC\times\cC$ such that it is graded-simple) and take the restriction to the subalgebra $\cC_0\times\cC_0$ (with the product given by the commutator, which is graded-simple too):
\[
\Gamma^1_{\cC_0\times\cC_0}(G,h,\overline{\gamma}):=\Gamma^1_{\cC\times\cC}(G,h,\overline{\gamma})\mid_{\cC_0\times\cC_0}.
\]
Then using the isomorphism
\begin{equation}\label{eq:iso.traceless}
\begin{array}{ccc}
\cC_0\times\cC_0 & \longrightarrow & (\cC_0\otimes \FF 1)\oplus (\FF 1 \otimes \cC_0)\\
(x,y) & \longmapsto & x\otimes 1 + 1 \otimes y
\end{array}
\end{equation}
we obtain a $G$-grading on $(\cC_0\otimes \FF 1)\oplus (\FF 1 \otimes \cC_0)$ (with the product given by the commutator and it is graded-simple), which we will denote by
\[
\Gamma^1_{(\cC_0\otimes \FF 1)\oplus (\FF 1 \otimes \cC_0)}(G,h,\overline{\gamma})
\]
where $\dg(x\otimes 1 + 1\otimes y)=g$ for $g\in G$ if $(x,y)\in \cC_0\times\cC_0$ is such that $\dg(x,y)=g$ in $\Gamma^1_{\cC_0\times\cC_0}(G,h,\overline{\gamma})$.
Finally, by Theorem \ref{th:iso.aut.schemes.cayley}, this last grading extends to a $G$-grading on $\cC\otimes \cC$ (with the usual product) which we denote by
\[
\Gamma^1_{\cC \otimes \cC}(G,h,\overline{\gamma}).
\]
Analogously, for an element $h$ of order 2 in $G$ and a subgroup $\overline{H} \subset \overline{G}=G/ \langle h \rangle$ isomorphic to $(\Z/2)^3$, we can construct from the grading $\Gamma_{\cC\times\cC}^2(G,h,\overline{H})$ (on $\cC\times\cC$ such that it is graded-simple) a grading on $\cC\otimes\cC$ denoted by
\[
\Gamma_{\cC\otimes\cC}^2(G,h,\overline{H}).
\]

The following result gives the classification of gradings, up to isomorphism, on $\cC\otimes \cC$ such that the graded subspace $(\cC_0\otimes\FF 1) \oplus (\FF 1\otimes\cC_0)$ is graded-simple.

\begin{corollary}\label{th:iso.CtensorC.graded-simple}
Let $\Gamma$ be a grading by a group $G$ on the tensor product $\cC\otimes\cC$ of two
Cayley algebras. Suppose that for the induced $G$-grading on the algebra $(\cC_0\otimes\FF 1) \oplus (\FF 1\otimes\cC_0)$ with the product given by the commutator (see Lemma \ref{le:grad.subspaces} b) and Definition \ref{df:structurable.algebra}) $(\cC_0\otimes\FF 1) \oplus (\FF 1\otimes\cC_0)$ is graded-simple. Then $\Gamma$ is isomorphic to either $\Gamma^1_{\cC\otimes\cC}(G,h,\overline{\gamma})$ or $\Gamma_{\cC\otimes\cC}^2(G,h,\overline{H})$ (for an element $h$ in $G$ of order 2, a triple $\overline{\gamma}=(\bar{g}_1, \bar{g}_2, \bar{g}_3)$ of elements in $\overline{G}=G/\langle h \rangle$ and a subgroup $\overline{H} \subset \overline{G}$ isomorphic to $(\Z/2)^3$). Moreover, no grading of the first type is isomorphic to one of the second type and
\begin{itemize}
\item $\Gamma^1_{\cC\otimes\cC}(G,h,\overline{\gamma})$ is isomorphic to $\Gamma^1_{\cC\otimes\cC}(G,h',\overline{\gamma}')$ if and only if $h=h'$ and $\overline{\gamma}\sim\overline{\gamma}'$.

\item $\Gamma_{\cC\otimes\cC}^2(G,h,\overline{H})$ is isomorphic to $\Gamma_{\cC\otimes\cC}^2(G,h',\overline{H}')$ if and only if $h=h'$ and $\overline{H}=\overline{H}'$.
\end{itemize}
\end{corollary}

\begin{proof}
Let $\Gamma$ be a $G$-grading on $\cC\otimes\cC$ such that for the induced $G$-grading $\Gamma_0$ on the algebra $(\cC_0\otimes\FF 1) \oplus (\FF 1\otimes\cC_0)$ with the product given by the commutator (see Lemma \ref{le:grad.subspaces} b) and Definition \ref{df:structurable.algebra}) $(\cC_0\otimes\FF 1) \oplus (\FF 1\otimes\cC_0)$ is graded-simple.

Using the isomorphism from Equation \eqref{eq:iso.traceless} we obtain a $G$-grading $\Gamma_{\cC_0\times\cC_0}$ on $\cC_0\times\cC_0$ isomorphic to $\Gamma_0$ where $\cC_0\times\cC_0$ is a graded-simple algebra (again with the product given by the commutator). Finally, by Remark \ref{re:C_0}, $\Gamma_{\cC_0\times\cC_0}$ induces a grading $\Gamma_{\cC\times\cC}$ on $\cC\times \cC$ (with the usual product) such that $\cC\times \cC$ is graded-simple. The result follows from Theorem \ref{th:iso.CxC.graded-simple}.
\end{proof}

Now we  generate gradings on $\cC\otimes\cC$ from gradings on $\cC\times\cC$ such that this cartesian product is not graded-simple, that is, such that $\cC\times 0$ and $0\times\cC$ are graded ideals.

Let $G$ be a group and let $\Gamma$ be a $G$-grading on $\cC\times\cC$ such that $\cC\times 0$ and $0\times\cC$ are graded ideals, then by \cite[Theorem 4.1]{CE18}
\[
\Gamma \simeq \Gamma^1\times_G\Gamma^2
\]
for some $G$-gradings $\Gamma^1$ and $\Gamma^2$ on $\cC$. Then we restrict the product $G$-grading to $\cC_0\times\cC_0$:
\[
\Gamma^1\times_G\Gamma^2 \vert_{\cC_0\times\cC_0}= \Gamma^1\vert_{\cC_0}\times_G\Gamma^2 \vert_{\cC_0}
\]
and using the isomorphism of Equation \eqref{eq:iso.traceless} we obtain a $G$-grading on $(\cC_0\otimes \FF 1)\oplus(\FF 1 \otimes\cC_0)$. Finally, by Theorem \ref{th:iso.aut.schemes.cayley}, this last grading extends to a $G$-grading on $\cC\otimes\cC$.

\vspace*{0.5cm}

Next result gives the classification of gradings, up to isomorphism, on $\cC\otimes \cC$ such that the graded subspace $(\cC_0\otimes\FF 1) \oplus (\FF 1\otimes\cC_0)$ is not graded-simple.

\begin{proposition}\label{pr:product.grading.isomorphism.CtensorC}
Let $G$ be a group and let $\Gamma$ be a $G$-grading on the tensor product of two Cayley algebras $\cC\otimes\cC$. Suppose that for the induced $G$-grading $\Gamma_0$ on the algebra $(\cC_0\otimes\FF 1) \oplus (\FF 1\otimes\cC_0)$ with the product given by the commutator  $(\cC_0\otimes\FF 1) \oplus (\FF 1\otimes\cC_0)$ is not graded-simple, i.e., $\cC_0\otimes\FF 1$ and $\FF 1\otimes\cC_0$ are $G$-graded ideals. By \cite[Theorem 4.1]{CE18} we have that $\Gamma_0$ is isomorphic to a product $G$-grading $\Gamma_0^1\times_G\Gamma_0^2$ for some $G$-gradings $\Gamma_0^1$ on $\cC_0\otimes\FF 1$ and $\Gamma_0^2$ on $\FF 1\otimes\cC_0$.

Let $\Gamma$ and $\Gamma'$ be $G$-gradings on $\cC\otimes\cC$ and let $\Gamma_0$ and $\Gamma_0'$ be the $G$-gradings induced on $(\cC_0\otimes\FF 1) \oplus (\FF 1\otimes\cC_0)$ by $\Gamma$ and $\Gamma'$, respectively. Let $\Gamma_0^1$ and $\Gamma_0'^1$ be $G$-gradings on $\cC_0\otimes\FF 1$ and let $\Gamma_0^2$ and $\Gamma_0'^2$ be $G$-gradings on $\FF 1\otimes\cC_0$ such that $\Gamma_0\simeq \Gamma_0^1\times_G\Gamma_0^2$ and $\Gamma'_0\simeq \Gamma_0'^1\times_G\Gamma_0'^2$.

Then, $\Gamma$ and $\Gamma'$ are isomorphic if and only if $\Gamma_0^1 \simeq \Gamma_0'^1$ and $\Gamma_0^2 \simeq \Gamma_0'^2$ or $\Gamma_0^1 \simeq \Gamma_0'^2$ and $\Gamma_0^2 \simeq \Gamma_0'^1$. \begin{flushright}
$\square$
\end{flushright}
\end{proposition}

\begin{df}\label{df:tensor_of_gradings}
Let $G$ and $H$ be groups. Let $\Gamma^1$ be a $G$-grading on an algebra $\cA$ and let $\Gamma^2$ be a $H$-grading on an algebra $\cB$. Recall that $\cA\otimes\cB$ has a natural $G\times H$-grading given by $(\cA\otimes\cB)_{(g,h)}= \cA_g\otimes\cB_h$. We  call this grading \textit{tensor product of $\Gamma^1$ and $\Gamma^2$} and denote it by $\Gamma^1\otimes\Gamma^2$.
\end{df}

Finally we obtain the fine gradings on $\cC\otimes\cC$ up to equivalence.

\begin{proposition}\label{pr:up_to_equivalence_CtensorC}
We have six different fine gradings, up to equivalence, on $\cC\otimes\cC$. Such gradings are in correspondence with the ones in Proposition \ref{pr:CxC.equivalence} and they are the following:
\begin{enumerate}
\item $\Gamma^1_{\cC} \otimes \Gamma^1_{\cC}$ by its universal group $\Z^2 \times \Z^2\simeq \Z^4$.

\item $\Gamma^1_{\cC} \otimes \Gamma^2_{\cC}$ by its universal group $\Z^2 \times (\Z/2)^3$.

\item $\Gamma^2_{\cC} \otimes \Gamma^2_{\cC}$ by its universal group $(\Z/2)^3 \times (\Z/2)^3 \simeq (\Z/2)^6$.

\item The grading $\Gamma_{\cC\otimes\cC}^1(\Z/2\times\Z^2,(\bar{1},0,0),((1,0),(0,1),(-1,-1)))$ on $\cC\otimes\cC$ by its universal group $\Z/2\times\Z^2$. This grading is generated by the following homogeneous components in $\cB:= (\cC_0\otimes\FF 1) \oplus (\FF 1\otimes\cC_0)$:
\[
\begin{array}{l}
\cB_{(\bar{0},0,0)}=\FF ((e_1-e_2)\otimes 1 + 1\otimes (e_1-e_2)),\\
\cB_{(\bar{1},0,0)}=\FF ((e_1-e_2)\otimes 1 - 1\otimes (e_1-e_2)),\\
\cB_{(\bar{0},g)}=\FF (x\otimes 1 + 1\otimes x),\\
\cB_{(\bar{1},g)}=\FF (x\otimes 1 - 1\otimes x),
\end{array}
\]
for $g \in \Z^2 \setminus \lbrace (0,0) \rbrace$ and $x\in \cC_g$ in $\Gamma^1_{\cC}$.

\item The grading $\Gamma_{\cC\otimes\cC}^2\Bigl(\bigl(\Z/2\bigr)^4,(\bar 1,\bar 0,\bar 0,\bar 0),\bigl(\Z/2\bigr)^3\Bigr)$ on $\cC\otimes\cC$ by its universal group $\bigl(\Z/2\bigr)^4$. This grading is generated by the following homogeneous components in $\cB:= (\cC_0\otimes\FF 1) \oplus (\FF 1\otimes\cC_0)$:
\[
\begin{array}{l}
\cB_{(\bar{0},g)}=\FF (x\otimes 1 + 1\otimes x),\\
\cB_{(\bar{1},g)}=\FF (x\otimes 1 - 1\otimes x),
\end{array}
\]
for $g \in (\Z/2)^3 \setminus \lbrace (\bar{0},\bar{0},\bar{0}) \rbrace$ and $x\in \cC_g$ in $\Gamma^2_{\cC}$.

\item The grading $\Gamma_{\cC\otimes\cC}^2(\Z/4\times(\Z/2)^2,(\widehat{2},\bar 0,\bar 0),(\Z/2)^3)$ on $\cC\otimes\cC$ by its universal group $\Z/4\times(\Z/2)^2$. Here we denote by $\widehat{m}$ the class of the integer $m$ modulo $4$ and restrict the usual notation $\bar{m}$ for the class of $m$ modulo $2$ and $\mathbf{i}$ denotes a square root of $-1$ in $\FF$. This grading is generated by the following homogeneous components in $\cB:= (\cC_0\otimes\FF 1) \oplus (\FF 1\otimes\cC_0)$:
\[
\cB_{(\widehat{m},\bar n, \bar p)}=\FF (x\otimes 1+\mathbf{i}^m \otimes x)
\]
for $(\pi(\widehat{m}),\bar n, \bar p) \in (\Z/2)^3$ and $x\in \cC_{(\pi(\widehat{m}),\bar n, \bar p)}$ in $\Gamma^2_{\cC}$.
\end{enumerate}
\end{proposition}

%%%%%%%%%%%%%%%%%%%%%%%%%%%%%%%%%%%%%%%%%%%%%%%%%%%%%%%%%%%%%%%%%%%%%%%%%%%%%%%%%%%%%%
%%%%%%%%%%%%%%%%%%%%    *SECTION*   SMIRNOV ALGEBRA               %%%%%%%%%%%%%%%%%%%%
%%%%%%%%%%%%%%%%%%%%%%%%%%%%%%%%%%%%%%%%%%%%%%%%%%%%%%%%%%%%%%%%%%%%%%%%%%%%%%%%%%%%%%
\section{Gradings on the Smirnov algebra} \label{section.smirnov}

In this section we first determine the automorphism group scheme of the Smirnov algebra $\smir$, and then we obtain a classification of the group gradings on $\smir$, in terms of the associated Cayley algebra.

We will only consider gradings on $\smir$ as an algebra with involution. Therefore, for any group grading on $\smir$, the projections $\pi_\cH \colon \smir \to \cH$, $h+s \mapsto h$ and $\pi_\cS \colon \smir \to \cS$, $h+s \mapsto s$ are homogeneous maps of trivial degree and the subspaces $\cH$ and $\cS$ are graded. The products $\odot$ and $[\cdot,\cdot]$ are also homogeneous since they are obtained using the projections of the product of $\smir$ in the subspaces $\cH$ and $\cS$.

We claim that for any grading on $\smir$, its universal group is abelian. Let $\Gamma$ be a $G$-grading on $\smir$ with $G = \Univ(\Gamma)$. Given a homogeneous basis $\{s_i\}_{i=1}^7$ of $\cS$, we have that $\pi_\cH(s_is_j) = s_i \times s_j$ and $\{ s_i \times s_j \med 1 \leq i \leq j \leq 7\}$ is a homogeneous basis of $\cH$. Therefore, $G$ is generated by the support of the subspace $\cS$. Since $\pi_\cH(s_is_j) = \pi_\cH(s_js_i) \neq 0$ with $\pi_\cH$ homogeneous of degree $0$, it follows that $\deg(s_i)\deg(s_j) = \deg(s_j)\deg(s_i)$, i.e., the elements of the support of $\cS$ commute. We conclude that $G$ is abelian.

From now on we will only consider gradings on $\smir$ by abelian groups, and for this reason the products of the groups will be denoted additively.

\begin{proposition}
The trace form $\tr \colon \smir \times \smir \to \FF$ is a nondegenerate symmetric bilinear form that is invariant (i.e., $\tr(\bar x ,\bar y) = \tr(x ,y)$ and $\tr(xy,z) = \tr(x,z\bar y)$ for $x,y\in\smir$) and homogeneous for any grading on $\smir$ (i.e., $\deg(x) + \deg(y) = 0$ whenever $\tr(x, y) \neq0$ for homogeneous elements $x, y\in\smir$).
\end{proposition}

\begin{proof}
In \cite[Eq.~(1.7)]{AF93b}, Allison and Faulkner proved that $\CxC$ has an invariant nondegenerate symmetric bilinear form given by
\begin{equation}\begin{aligned}
\chi\colon (\CxC) \times (\CxC) & \longrightarrow\FF \\
(a \otimes b, c \otimes d) & \longmapsto n(a,c)n(b,d),
\end{aligned}\end{equation}
which in turn restricts to an invariant nondegenerate symmetric bilinear form $\chi|_{\smirCxC}$ of $\smirCxC$. Note that with the corresponding identification through the isomorphism $\psi$ in \eqref{isomSmirnovConstructions}, $\chi|_{\smirCxC}$ is proportional to the trace $\tr$ of $\smir$, and therefore $\tr$ satisfies the same properties. It remains to prove that $\tr$ is homogeneous for any grading.

Fix a $G$-grading $\Gamma \colon \smir = \bigoplus_{g\in G} \smir_g$. Since $\tr(x, y) = \tr(x \bar y)$, it suffices to prove that $\tr(\smir_g) = 0$ for each $0 \neq g\in G$. Note that $\tr(\cS) = 0$, so we only need to prove that $\tr(\smir_g \cap \cH) = 0$ if $0 \neq g\in G$. Recall that if $\{s_i\}_{i=1}^7$ is a homogeneous basis of $\cS$ then $\{ s_i \times s_j \med 1\leq i \leq j \leq 7 \}$ is a homogeneous basis of $\cH$ and $\deg(s_i \times s_j) = \deg(s_i) + \deg(s_j)$. By \eqref{eqf}, \eqref{eqfn}, and the fact that $[\cdot,\cdot]$ is homogeneous it follows that $\deg(s)+\deg(t) = 0$ for any homogeneous elements $s,t\in\cS$ such that $n(s,t) \neq 0$. In other words, $\deg(s \times t) = 0$ for any homogeneous elements $s,t\in\cS$ such that $\tr(s \times t)\neq 0$, and therefore $\tr$ is homogeneous.
\end{proof}

\begin{remark}
Consider the linear map $\pi_1 \colon \smir \to \smir$ determined by $(s_1 \otimes s_2) + s \mapsto \frac{-1}{2}n(s_1,s_2)1 = \frac{1}{16}\tr(s_1,s_2)1$ for $s_1,s_2,s\in\cS$. The map $\pi_1$ is homogeneous of trivial degree because the trace form $\tr$ is homogeneous. The identity elements of the algebras $\cC$ and $\smir$ can be identified, so $\im \pi_1 = \FF1 \subseteq\cC$. Therefore, the product $\cdot$ of the Cayley algebra $\cC$ can be recovered from the product of $\smir$ as follows:
\begin{equation} \label{eqRecoverProduct}
s_1 \cdot s_2 = \pi_1(s_1s_2) + \pi_\cS(s_1s_2)
\end{equation}
for any $s_1,s_2 \in\cS$, and obviously $1 \cdot x = x = x \cdot 1$ for $x\in\cC$.
\end{remark}

\begin{theorem} \label{smirnovschemetheorem}
The automorphism group schemes $\AAut(\cC)$ and $\AAut(\smir, -)$ are isomorphic.
\end{theorem}
\begin{proof}
First, note that the skew subspace $\cS = \cS(\smir, -)$ can be identified with the skew subspace $\cC_0$ of $\cC$, and recall that $(\cS, [\cdot,\cdot])$ is a Malcev algebra. There is an isomorphism of automorphism group schemes $\AAut(\cC) \simeq \AAut(\cS, [\cdot,\cdot])$ given by the restriction map $\phi \mapsto \phi|_{\cC_0}$. Since $\cS$ generates $\smir$, the natural map $\AAut(\smir, -) \to \AAut(\cS, [\cdot,\cdot])$ given by the restriction $\varphi\mapsto\varphi|_\cS$ is an embedding. Also, the extension map $\AAut(\cC)\to\AAut(\smir, -)$ is an embedding too. Since the composition
$$ \AAut(\smir, -) \longrightarrow \AAut(\cS, [\cdot,\cdot]) \simeq \AAut(\cC) \longrightarrow \AAut(\smir, -) $$
is the identity map, it follows that all these maps are isomorphisms and $\AAut(\cC) \simeq \AAut(\smir, -)$.
\end{proof}

\begin{remark}
Note that $\AAut(\cC) = \AAut(\cC, -)$. The isomorphism in Theorem~\ref{smirnovschemetheorem} and its inverse are given by the extension map
$$\text{Ext} \colon \AAut(\cC) \to \AAut(\smir, -)$$
and the restriction map
$$\text{Res} \colon \AAut(\smir, -) \to \AAut(\cC).$$
That is, for any commutative associative unital $\FF$-algebra $R$,
\begin{align*}
\text{Ext}_R \colon & \Aut_R(\cC \otimes R) \longrightarrow \Aut_R(\smir \otimes R,-) \\
& \phi \mapsto \text{Ext}_R(\phi) \colon \widetilde{s} \mapsto \phi(\widetilde{s}), \; \widetilde{s_1} \times \widetilde{s_2} \mapsto \phi(\widetilde{s_1}) \times \phi(\widetilde{s_2}),
\end{align*}
and
\begin{align*}
\text{Res}_R \colon & \Aut_R(\smir \otimes R,-) \longrightarrow \Aut_R(\cC \otimes R) \\
& \psi \mapsto \text{Res}_R(\psi) \colon 1 \mapsto 1, \; \widetilde{s} \mapsto \psi(\widetilde{s}),
\end{align*}
for each $\widetilde{s},\widetilde{s_1},\widetilde{s_2}\in\cS \otimes R$.
\end{remark}

\begin{remark}
Recall that in \cite[Theorem 2.6]{AF93b} it was proven that the twisted forms of $(\smir, -)$ are the algebras $(\TT(\widetilde{\cC}), -)$ where $\widetilde{\cC}$ is a twisted form of $\cC$. Note that, as in Remark~\ref{re.forms.CxC}, the above theorem allows to prove the same result in a different way.

Indeed, recall from \cite{Wat79} that the isomorphism classes of twisted forms of an algebra $\cA$ can be identified with the elements of the set $H^1(\bar\FF/\FF, \AAut(\cA))$. By Theorem~\ref{smirnovschemetheorem}, we have an isomorphism $\AAut(\smir,-) \simeq \AAut(\cC)$ which in turn produces a bijection between the cohomology sets $H^1(\bar\FF/\FF, \AAut(\smir, -)) \to H^1(\bar\FF/\FF, \AAut(\cC))$. Therefore, there is a natural correspondence between the twisted forms of $\smir$ and the twisted forms of $\cC$. Furthermore, Equation~\eqref{eqRecoverProduct} allows to recover the product of $\cC$ from the product of $\smir$, so it is clear that twisted forms of $\smir$ are as stated above.

Note that the Cayley algebra, up to isomorphism, has exactly two real forms: the split Cayley algebra $\cC_s$ and the division Cayley algebra $\OO$. Therefore there are exactly, up to automorphism, two real forms of the Smirnov algebra: $\TT(\cC_s)$ and $\TT(\OO)$.
\end{remark}

\begin{corollary} \label{clasif-grads-smirnov}
There is a correspondence between the gradings on $(\cC,-)$ and the gradings on $(\smir,-)$ that preserves universal groups, equivalence classes, isomorphism classes, and the Weyl groups of the gradings.
\end{corollary}
\begin{proof}
This is consequence of Theorem~\ref{smirnovschemetheorem} and \cite[Theorems 1.38 and 1.39]{EKmon}.
\end{proof}

\bigskip

We will now describe more explicitly how to construct the gradings on the Smirnov algebra with both constructions of the Smirnov algebra.

\begin{example} Let $\Gamma$ be a $G$-grading on $\cC$ with degree map $\deg_\cC$. Then $\FF 1$ and the skew subspace $\cS$ of $\cC$ are graded. We can identify the skew subspaces of $\cC$ and $\smir$. We claim that we have a $G$-grading $\widetilde{\Gamma}$ on $\smir$ with degree map $\deg$ determined by
\begin{equation*}
\deg(1_\smir) = 0, \quad  \deg(s) = \deg_\cC(s), \quad \deg(s_1 \times s_2) = \deg_\cC(s_1) + \deg_\cC(s_2)
\end{equation*}
for any homogeneous elements $s, s_1, s_2 \in \cS$ in $\Gamma$. This follows from the definition of the product of $\smir$ and the fact that the bilinear form $n$ of $\cC$ is graded for any grading on $\cC$ (see \cite[Eq.(4.10)]{EKmon}).

\smallskip

In order to construct the equivalent grading with the second construction $\smirCxC$ of the Smirnov algebra it suffices to apply the isomorphism \eqref{isomSmirnovConstructions} to the grading $\widetilde{\Gamma}$ given above. The degree map is now determined by
\begin{equation*}
\deg(s \otimes 1 + 1 \otimes s) = \deg_\cC(s), \;
\deg(s_1 \otimes s_2 + s_2 \otimes s_1 - n(s_1,s_2)1 \otimes 1) = \deg_\cC(s_1) + \deg_\cC(s_2)
\end{equation*}
for $s,s_1,s_2\in\cS$ homogeneous in $\Gamma$.
\end{example}

\begin{theorem}
If the base field $\FF$ is algebraically closed of characteristic different from $2$, then there are exactly two fine involution preserving gradings, up to equivalence, on the Smirnov algebra. These have universal groups $\ZZ^2$ and $(\ZZ/2)^3$.
\end{theorem}
\begin{proof}
This follows from the classification of fine gradings on $\cC$ and Corollary~\ref{clasif-grads-smirnov}.
\end{proof}

%%%%%%%%%%%%%%%%%%%%%%%%%%%%%%%%%%%%%%%%%%%%%%%%%%%%%%%%%%%%
%%                                                        %%
%%          GRADUACIONES EN ALG DE LIE ASOCIADAS          %%
%%                                                        %%
%%%%%%%%%%%%%%%%%%%%%%%%%%%%%%%%%%%%%%%%%%%%%%%%%%%%%%%%%%%%
\section{Induced gradings} \label{section.induced}

As an application of our results above, we can use several constructions to induce gradings from structurable algebras to Lie algebras. We will now recall several of these constructions.

\medskip
%% KANTOR CONSTRUCTION:

We will first recall the $5$-graded Lie algebra obtained with the Kantor construction from a structurable algebra (\cite[Theorem 3]{All79}), or more generally, from a Kantor pair (\cite[\S3--4]{AF99}). The Kantor construction generalizes the Tits-Kantor-Koecher (TKK) construction for Jordan systems.

\begin{df}
A {\em Kantor pair} (or {\em generalized Jordan pair of second order} \cite{F94, AF99}) is a pair of vector spaces $\cV = (\cV^+, \cV^-)$ and a pair of trilinear products $\cV^\sigma\times\cV^{-\sigma}\times\cV^\sigma\to\cV^\sigma$ (with $\sigma\in\{+, -\}$), denoted by $\{x,y,z\}^\sigma$, satisfying the identities:
\begin{align}
& [V^\sigma_{x,y}, V^\sigma_{z,w}] = V^\sigma_{V^\sigma_{x,y}z,w} - V^\sigma_{z,V^{-\sigma}_{y,x}w}, \\
& K^\sigma_{K^\sigma_{x,y}z,w} = K^\sigma_{x,y}V^{-\sigma}_{z,w} + V^\sigma_{w,z}K^\sigma_{x,y},
\end{align}
where $V^\sigma_{x,y}z = U^\sigma_{x,z}(y):= \{x,y,z\}^\sigma$, $U^\sigma_x := U^\sigma_{x,x}$ and $K^\sigma_{x,y}z = K^\sigma(x,y)z := \{x,z,y\}^\sigma - \{y,z,x\}^\sigma$. The map $V^\sigma_{x,y}$ is also denoted by $D^\sigma_{x,y}$ or $D^\sigma(x,y)$ (because $(V^+_{x,y},-V^-_{y,x})$ is a derivation of the Kantor pair). Recall that if $(\cA, -)$ is a structurable algebra, then $(\cA, \cA)$, with two copies of the triple product of $\cA$, defines a Kantor pair.
\end{df}

Consider the vector space
\begin{equation}
\kan(\cV) \coloneqq \kan(\cV)^{-2} \oplus \kan(\cV)^{-1} \oplus \kan(\cV)^{0} \oplus \kan(\cV)^{1} \oplus \kan(\cV)^{2},
\end{equation}
where
\begin{align*}
& \kan(\cV)^{-2} = \lspan\left\{\left(\begin{matrix} 0 & K(\cV^-,\cV^-) \\ 0 & 0 \end{matrix}\right)\right\}, \quad
\kan(\cV)^{-1} = \left(\begin{matrix} \cV^- \\ 0 \end{matrix}\right), \\
& \kan(\cV)^{0} = \lspan\left\{\left(\begin{matrix} D(x^-,x^+) & 0 \\ 0 & -D(x^+,x^-) \end{matrix}\right)
  \med x^\sigma\in\cV^\sigma \right\}, \\
& \kan(\cV)^{1} = \left(\begin{matrix} 0 \\ \cV^+ \end{matrix}\right), \quad
\kan(\cV)^{2} = \lspan\left\{\left(\begin{matrix} 0 & 0 \\ K(\cV^+,\cV^+) & 0 \end{matrix}\right)\right\}.
\end{align*}
Then, the vector space
\begin{align*}
\mathfrak{S}(\cV) \coloneqq& \kan(\cV)^{-2} \oplus \kan(\cV)^0 \oplus \kan(\cV)^{2} \\
=& \lspan\left\{\left(\begin{matrix} D(x^-,x^+) & K(y^-,z^-) \\ K(y^+,z^+) & -D(x^+,x^-) \end{matrix}\right)
  \med x^\sigma,y^\sigma,z^\sigma\in\cV^\sigma \right\}
\end{align*}
is a subalgebra of the Lie algebra $\End\begin{pmatrix} \cV^- \\ \cV^+ \end{pmatrix}
= \begin{pmatrix} \End(\cV^-) & \Hom(\cV^+,\cV^-) \\ \Hom(\cV^-,\cV^+) & \End(\cV^+) \end{pmatrix}$,
with the commutator product. The product of $\kan(\cV)$ is defined by
\begin{align*}
& [A,B] = AB-BA, \qquad
[A, \begin{pmatrix} x^- \\ x^+ \end{pmatrix} ] = A \begin{pmatrix} x^- \\ x^+ \end{pmatrix}, \\
& [\begin{pmatrix} x^- \\ x^+ \end{pmatrix}, \begin{pmatrix} y^- \\ y^+ \end{pmatrix}] =
\left(\begin{matrix} D(x^-,y^+)-D(y^-,x^+) & K(x^-,y^-) \\ K(x^+,y^+) & -D(y^+,x^-)+D(x^+,y^-) \end{matrix}\right)
\end{align*}
for $A, B\in \mathfrak{S}(\cV)$ and $x^\sigma, y^\sigma \in \cV^\sigma$ for $\sigma = \pm$.

Then, $\kan(\cV)$ becomes a Lie algebra, called the \emph{Kantor Lie algebra} of $\cV$. The $5$-grading is a $\ZZ$-grading which is called the \emph{standard grading} of $\kan(\cV)$; we will also refer to it as the \emph{main grading} of $\kan(\cV)$. The subspaces $\kan(\cV)^1$ and $\kan(\cV)^{-1}$ are usually identified with $\cV^+$ and $\cV^-$, respectively. The Kantor construction of a structurable algebra is defined as the Kantor construction of the associated Kantor pair.

Let $\cA$ be a structurable algebra. Recall that $\nu(x^-,x^+)\coloneqq(D_{x^-,x^+},-D_{x^+,x^-})$ is a derivation called \emph{inner derivation} associated to $(x^-,x^+)\in\cV^-\times\cV^+$. The \emph{inner structure algebra} of $\cA$ is the Lie algebra $\mathfrak{innstr}(\cA)=\lspan\{\nu(x,y)\med x,y\in\cA\}$. Let $L_x$ denote the left multiplication by $x\in\cA$ and write $\cS=\cS(\cA)$. Then, the map $\cS\to L_\cS$, $s\mapsto L_s$, is a linear monomorphism, so we can identify $\cS$ with $L_\cS$. Also, note that the map $\cA\times\cA\to\cS$ given by $\psi(x,y)\coloneqq x\bar y-y\bar x$ is an epimorphism (because $\psi(s,1)=2s$ for $s\in\cS$). By \cite[(1.3)]{AF84}, we have the identity $L_{\psi(x,y)}=U_{x,y}-U_{y,x}=K(x,y)$ for all $x,y\in\cA$. As a consequence of this, in the Kantor construction of $\cV$ we can identify the subspaces $\kan(\cV)^{\sigma2}$ with $L_\cS$, and also with $\cS$. Then, the main grading of $\kan(\cA)$ can be written as follows:
\begin{equation}
\kan(\cA) = \cS^- \oplus \cA^- \oplus \mathfrak{innstr}(\cA) \oplus \cA^+ \oplus \cS^+.
\end{equation}

\bigskip
 
Let $\Gamma$ be a $G$-grading on $(\cA,-)$ (i.e., an involution preserving $G$-grading on $\cA$). Note that $\Gamma$ induces a $G$-grading on $\Der(\cA,-)$ and also on $\mathfrak{innstr}(\cA)$. Moreover, $\Gamma$ extends to a $\Z\times G$-grading on $\kan(\cA)$ by means of
\begin{equation}
\dg s^\pm= (\pm 2,\dg_{\Gamma}s),
\quad \dg a^\pm =(\pm 1,\dg_{\Gamma} a),
\quad \dg f= (0, \dg_{\Gamma}f)
\end{equation}
for homogeneous elements $s^\pm\in \cS^\pm$, $a^\pm\in \cA^\pm$ and $f\in \mathfrak{innstr}(\cA)$, where $\dg_{\Gamma}a$ and $\dg_{\Gamma}s$ denote the degrees in $\Gamma$ and $\dg_{\Gamma}f$ the induced degree in $\mathfrak{innstr}(\cA)$.

\bigskip

Recall that
$\kan(\smir) = \mathfrak{e}_7$,
$\kan(\cC\otimes\cC) = \mathfrak{e}_8$,
$\kan(\cC\otimes\cH) = \mathfrak{e}_7$,
$\kan(\cC\otimes\cK) = \mathfrak{e}_6$, and
$\kan(\cC\otimes\FF) = \mathfrak{f}_4$.

The fine gradings on $\smir$, by $\ZZ^2$ and $(\ZZ/2)^3$ respectively, induce a $\ZZ^3$-grading and a $\ZZ \times (\ZZ/2)^3$-grading on $\mathfrak{e}_7$.

The fine gradings on $\cC\otimes\cC$ induce gradings on $\mathfrak{e}_8$ by the groups:
$\Z \times \Z^4$, $\Z^3 \times (\Z/2)^3$, $\Z \times (\Z/2)^6$, $\Z^3\times\Z/2$, $\Z\times\bigl(\Z/2\bigr)^4$ and $\Z\times \Z/4\times(\Z/2)^2$.

The fine gradings on $\cC\otimes\cH$ induce gradings on $\mathfrak{e}_7$ by the groups:
$\Z\times(\Z/2)^5$,  $\Z^2\times(\Z/2)^3$,  $\Z^3\times(\Z/2)^2$ and $\Z^4$.

The fine gradings on $\cC\otimes\cK$ induce gradings on $\mathfrak{e}_6$ by the groups:
$\Z\times(\Z/2)^4$ and $\Z^3\times\Z/2$.

The fine gradings on $\cC\otimes\FF$ induce gradings on $\mathfrak{f}_4$ by the groups: $\ZZ^3$ and $\ZZ \times (\ZZ/2)^3$.

\bigskip
%% SMIRNOV CONSTRUCTION:

Recall now from \cite{AF93a} the construction of the {\em Steinberg unitary Lie algebra} $\mathfrak{stu}_3(\cA, -)$ for a unitary nonassociative algebra with involution $(\cA, -)$, which is generated by the symbols $u_{ij}(x)$ for $1\leq i\neq j \leq 3$, $x\in\cA$, subject to the relations
\begin{align*}
& u_{ij}(x) = u_{ji}(-\bar x), \\
& x \mapsto u_{ij}(x) \quad \text{is linear}, \\
& [u_{ij}(x), u_{jk}(y)] = u_{ik}(xy) \quad \text{for distinct $i, j, k$.}
\end{align*}
Also, by \cite[Lemma~1.1]{AF93a},
$$ \mathfrak{stu}_3(\cA, -) = \mathfrak{s} \oplus u_{12}(\cA) \oplus u_{23}(\cA) \oplus u_{31}(\cA),  $$
where $\mathfrak{s} = \sum_{i<j}[u_{ij}(\cA), u_{ij}(\cA)]$, and this decomposition defines a $(\ZZ/2)^2$-grading on $\mathfrak{stu}_3(\cA, -)$. Moreover, any $G$-grading on $\cA$ induces a $(\ZZ/2)^2 \times G$-grading on $\mathfrak{stu}_3(\cA, -)$.

\smallskip

Note that the next result appears in the literature (\cite[Section~6]{AEK14}) for a different structurable algebra known as the Brown algebra, and we will use the same arguments in our proof:
\begin{proposition}
Assume that $\chr(\FF)\neq 2,3,5$. Let $\cA$ be one of the structurable algebras $\cC\otimes H$, with $H$ a Hurwitz algebra, or $\smir$. Then $\kan(\cA, -)$ and $\mathfrak{stu}_3(\cA, -)$ are isomorphic.
\end{proposition}
\begin{proof}
Note that $\kan(\cA, -)$ is simple in all cases. There is an isomorphism between the quotient of $\mathfrak{stu}_3(\cA, -)$ by its center and $\kan(\cA, -)$ (see \cite{AF93a}, \cite{EO07}).  Since $\chr(\FF)\neq 2,3,5$, the Killing form of $\kan(\cA, -)$ is nondegenerate, so it has no nontrivial central extensions. Hence, the center of $\mathfrak{stu}_3(\cA, -)$ is trivial and the result follows.
\end{proof}

Finally, we give a list of the gradings on $\mathfrak{stu}_3(\cA, -)$ that are induced by the fine gradings on the algebras considered in this paper.

The fine gradings on $\smir$ induce gradings on $\mathfrak{e}_7$ by the groups: $\ZZ^2 \times (\ZZ/2)^2$ and $ (\ZZ/2)^5$.

The fine gradings on $\cC\otimes\cC$ induce gradings on $\mathfrak{e}_8$ by the groups: $\Z^4 \times (\ZZ/2)^2$, $\Z^2 \times (\Z/2)^5$, $ (\Z/2)^8$, $\Z^2\times (\ZZ/2)^3$, $\bigl(\Z/2\bigr)^6$ and $ \Z/4\times(\Z/2)^4$.

The fine gradings on $\cC\otimes\cH$ induce gradings on $\mathfrak{e}_7$ by the groups: $(\Z/2)^7$,  $\Z\times(\Z/2)^5$,  $\Z^2\times(\Z/2)^4$ and $\Z^3\times(\Z/2)^2$.

The fine gradings on $\cC\otimes\cK$ induce gradings on $\mathfrak{e}_6$ by the groups: $(\Z/2)^6$ and $\Z^2\times (\ZZ/2)^3$.

The fine gradings on $\cC\otimes\FF$ induce gradings on $\mathfrak{f}_4$ by the groups: $\ZZ^2 \times (\ZZ/2)^2$ and $ (\ZZ/2)^5$.

\bigskip

%% Acknowledgements
\textbf{Acknowledgements} Both authors are very grateful to Alberto Elduque for supervision of this work.
Thanks are also due to the anonymous referee, for his corrections and suggestions.
Note that a big part of this paper is part of the PhD Thesis of Alejandra S. C\'ordova-Mart\'inez.

%%%%%%%%%%%%%%%%%%%%%%%%%%%%%%%%%%%%%%%%%%%%%%%%%%%%%%%%%%%%%%%%% Referencias

%%%%%%%%%%%%%%%%%%%%%%%%%%%%%%%%%%%%%%%%%%%%%%%%%%%%%%%%%%%%%%%%%%%%%%%%%%%%%%%%%%
%%%%%%%%%%%%%%%%%%%%%%%%%%%%%%%%%%%%%%%%%%%%%%%%%%%%%%%%%%%%%%%%%%%%%%%%%%% FIN
\end{document}